\documentclass{article}
\usepackage[utf8]{inputenc}
\input{preamble}

\title{Dissipation Properties of Transport Noise in the Two-layer Quasi-Geostrophic Model}
\author{%
  Giulia Carigi
  \and Eliseo Luongo\footnote{Corresponding author: eliseo.luongo@sns.it}
  }
\date{26/09/2022}
\begin{document}
\maketitle
\begin{abstract}
   A stochastic version of the two--layer quasi--geostrophic model (2LQG) with multiplicative transport noise is analysed. This popular intermediate complexity model describes large scale atmosphere and ocean dynamics at the mid-latitudes. The transport noise, which acts on both layers, accounts for the unresolved small scales. After establishing the well-posedness of the perturbed equations, we show that, under a suitable scaling of the noise, the solutions converge to the deterministic 2LQG model with enhanced dissipation. Moreover, these solutions converge to the deterministic stationary ones on the long time horizon.
   
\end{abstract}


\section{Introduction}
 
The quasi-geostrophic model, first mathematically described by Charney \cite{Charney48} in the first half of the XX century, is an approximation of the three dimensional Navier-Stokes equation in vorticity formulation which captures the large-scale phenomena of atmosphere and ocean dynamics.
Quasi-geostrophic models with several layers, like the one considered in this article, are particularly suitable to describe baroclinic instabilities, a crucial mechanism behind most common weather patters at the mid-latitudes.
In fact the quasi--geostrophic equations is a popular model in theoretical meteorology given the richness of phenomena it can describe (see e.g. \cite{Vallis06,Pedlosky13}). From a mathematical perspective, it can be seen as a model of intermediate complexity between two and three dimensional Navier--Stokes equation. The model has been shown to be well-posed and to exhibit a global attractor in its deterministic version (see \cite{Bernier94}) and in a stochastic version with additive noise \cite{Chueshov01_Proba,thesis}. In this paper a stochastically perturbed version of the two--layer quasi--geostrophic (2LQG) model with multiplicative noise of transport type on both layers is studied and shown to be well-posed.

The nature of this perturbation is crucial both from a physical interpretation of the results and for the novelty of mathematical tools it requires. 
Recent developments on the physical justification for transport noise in fluid dynamics models include \cite{Holm_2015}. There the transport noise is systematically introduced in models relevant for geophysical applications, including the quasi-geostrophic model, in such a way to retain conservation laws crucial for the description of fluids, in particular circulation. 
Recently \cite{FlandoliPappalettera2022} proposed a rigorous interpretation and justification of transport noise as additive noise on smaller scales. 
For more on the interpretation of transport noise in stochastic partial differential equation refer to \cite{FlandoliPappalettera2022} and references therein. 

From a mathematical perspective the literature on transport noise for fluid dynamics models is vast, especially for Euler and Navier-Stokes equation. Most relevant for this work are \cite{flandoli2022eddy,flandoli2021quantitative}. In these works the authors make use a suitable scaling, first introduced in \cite{galeati2020convergence}, to show that weak solutions of a stochastic Euler model converge to weak solutions of the deterministic Navier-Stokes equation, and provide a quantitative estimate of the rate of convergence. Similarly, for the 2LQG model we will see that the transport noise provides an enhancement of eddy dissipation effects. These results highlight the regularization action of the transport noise, which, despite being energy-preserving in general, acts as a dissipative force. On the relevance of transport noise for dissipation and mixing properties we refer as well to \cite{Flandoli_2021, dong2022dissipation}. 

Thanks to dissipation properties of the transport noise we will also show that the solution of the 2LQG model approaches the deterministic stationary solution for large times. The well-posedness of the stationary system associated to the 2LQG equations would require in general a large model viscosity, hard to justify from a physical perspective. However, we will show that an appropriate choice of the transport noise ensures stability on the long run of the solutions. This results extends those obtained for the heat equations in \cite{flandoli2022heat} to a nonlinear system like the quasi-geostrophic model here under analysis. We expect this approach to generalise to the two--dimensional Navier-Stokes equations and other models with similar structure. 

\subsection{The Model and the Main Results}\label{sec model}
The two-layer quasi-geostrophic model (2LQG) is 
one of the most used model to describe the motion of atmosphere as well as of the ocean at the mid-latitudes. In particular it captures the large scale dynamics of two layers of fluid of fixed height $h_1$, $h_2$ respectively and with density $\rho_1$ and $\rho_2$ with $\rho_1<\rho_2$. We consider the $\beta$-plane approximation to the Coriolis effect for which the Coriolis parameter $f_c$ can be expressed as 
\(
 f_c(y) = f_0 + \beta y,
\)
with $f_0$ and $\beta$ assigned positive constants. In the model we include the effect of eddy viscosity on both layers, of the bottom friction on the second layer, to account for the interaction with the Eckmann layer, and a deterministic additive forcing on the first layer, for example to account for the wind forcing on the upper ocean. In order to give a mathematical formulation of the model let us introduce a spatial domain $\cD$, squared domain $\cD = [0,L]\times [0,L]\subset \R^2$ (where $L$ will be e.g. $10^5$ m for the ocean). 

Consider the following equations for the variables $\bfpsi(t, \x) = (\psi_1(t, \x), \psi_2(t, \x))^t$, streamfunction of the fluid, and $\q(t,\x) = (q_1(t, \x), q_2(t, \x))^t $, the so-called quasi--geostrophic (QG) potential vorticity
\begin{align}
\label{eq:QG_deterministic}
\begin{split}
    dq_1 + \left(\nabla^{\perp}\psi_1 \cdot \nabla q_1\right) dt= \left(\nu\Delta^2 \psi_1 -\beta \partial_x \psi_1+ F(t) \right)dt \\
    dq_2 + \left(\nabla^{\perp}\psi_2 \cdot \nabla q_2\right) dt= \left(\nu\Delta^2 \psi_2 -\beta \partial_x \psi_2 - r\Delta \psi_2 \right)dt .
\end{split}
\end{align}
Here $\nu>0$ is the eddy viscosity parameter, $r>0$ accounts for the bottom friction and $F(t)$ is a deterministic forcing with zero spatial averages, namely 
\begin{equation}
\label{zero average}
    \int_{\cD} F(t, \x)\, d\x=0 \fa t\geq 0.
\end{equation}
The QG potential vorticities $\q$ and the streamfunction $\bfpsi$ are linked by the equations
\begin{equation}
\label{eq:relation_q_psi}
    \begin{split}
        q_1=\Delta \psi_1+S_1(\psi_2-\psi_1)\\
        q_2=\Delta \psi_2+S_2(\psi_1-\psi_2).
    \end{split}
\end{equation}
where $S_1, S_2$ are positive constants such that 
\begin{equation}
\label{eq:defS}
    h_1 S_1 = h_2 S_2 =:S.
\end{equation}

As stated in the introduction, the goal of the present work is to study the eddy dissipation properties of transport noise for this model. 
There are several motivations to consider transport noise, as the effect of small scales on large scales in fluid dynamical problems, see \cite{flaWas}, \cite{Holm_2015} for several discussion on this topic. Loosely speaking, small scale transport noise produces in the limit an extra dissipative term, which can be called eddy dissipation. 
There is an extended literature devoted to these kind of topics both in the endogenous and the exogenous case. See for example \cite{flandoli2021quantitative}, \cite{flandoli2022eddy}, \cite{flandoli2022heat}, \cite{dong2022dissipation}.


 Let us now introduce the stochastic perturbation we will consider. Let  $\left(\Omega,\mathcal{F},\mathcal{F}_t,\mathbb{P}\right)$ be a filtered probability space. Let $\{W^{j,k}_t\}_{k=1}^K,\ j\in\{1,\ 2\}$ be two sequences of real Brownian motions adapted to $\mathcal{F}_t$ and consider two corresponding sequences of divergence free vector fields $\{\bm{\sigma}_{j,k}\}_{k=1}^K\subseteq C^{\infty}(\mathcal{D};\mathbb{R}^2), \ j\in\{1,\ 2\}$. 
 In general, less can be required on the regularity of the coefficients $\bm{\sigma}_{j,k}$, but it is not our goal to stretch the boundaries of regularity of the noise as it is not critical to obtain the desired final result.
Now we can consider the following stochastically perturbed two-layer quasi-geostrophic model
\begin{align}
\label{eq:QG_strato}
\begin{split}
    dq_1 + \left(\nabla^{\perp}\psi_1 \cdot \nabla q_1\right)dt= \left(\nu\Delta^2 \psi_1 -\beta \partial_x \psi_1+ F \right)dt + \sum_{k=1}^K\left(\bm{\sigma}_{1,k} \circ dW^{1,k} \right)\cdot \nabla q_1 \\
    dq_2 + \left(\nabla^{\perp}\psi_2 \cdot \nabla q_2\right)dt= \left(\nu\Delta^2 \psi_2 -\beta \partial_x \psi_2 - r\Delta \psi_2 \right)dt + \sum_{k=1}^K\left(\bm{\sigma}_{2,k} \circ dW^{2,k}\right)\cdot \nabla q_2
\end{split}
\end{align}
which in It\^{o} formulation reads
\begin{align}
\label{eq:QG_ito}
\begin{split}
    dq_1 + \left(\nabla^{\perp}\psi_1\cdot \nabla q_1\right) dt = \left(\nu\Delta^2 \psi_1 - \beta \partial_x \psi_1 + F \right)\, dt + \sum_{k=1}^K\bm{\sigma}_{1,k}\cdot \nabla q_1\,  dW^{1,k}+\tfrac{1}{2}\sum_{k=1}^K\bm{\sigma}_{1,k} \cdot \nabla (\bm{\sigma}_{1,k}\cdot \nabla q_1)\, dt\\
    dq_2 + \left(\nabla^{\perp}\psi_2\cdot \nabla q_2\right) dt = \left(\nu\Delta^2 \psi_2 -\beta \partial_x \psi_2 - r\Delta \psi_2\right)\, dt + \sum_{k=1}^K\bm{\sigma}_{2,k}\cdot \nabla q_2 \, dW^2 +\tfrac{1}{2}\sum_{k=1}^K\bm{\sigma}_{2,k} \cdot \nabla (\bm{\sigma}_{2,k}\cdot \nabla q_2)\, dt.
    \end{split}
\end{align}
We assume periodic boundary conditions for the streamfunction $\bfpsi$ in both directions and that 
\begin{equation*}
    \int_{\cD} \bfpsi(t, \x) \, d\x  = 0.
\end{equation*}


Let us set some notation before stating the main contributions of this work. Let $(H^k(\mathcal{D}), \lVert\cdot\rVert_{H^k}), k\in\mathbb{R}$ be the standard Sobolev spaces of $L$-periodic functions satisfying condition \eqref{zero average}. We will denote by $\langle \cdot,\cdot\rangle_{H^k}$ the corresponding scalar products. With a slight abuse of notation, for $k>0$ we denote the dual pairing with 
\begin{equation*}
    \langle T, \varphi \rangle_{H^{-k},H^k} = T(\varphi) \fa T\in H^{-k},\, \varphi\in H^k.
\end{equation*}
In case of $k=0$ we will write $L^2(\mathcal{D})$ instead of $H^0(\mathcal{D})$ and we will neglect the subscript in the notation for the norm.
Similarly, we introduce the Sobolev spaces of zero mean vector fields
\begin{align*}
    \mathbf{H}^k= \{(u_1,u_2)^t:\ u_1,u_2\in H^k(\mathcal{D}) \},\ \langle \bfu,\bfv\rangle_{\mathbf{H}^k}=\langle u_1,v_1\rangle_{H^k}+\langle u_2,v_2\rangle_{H^k},\, \text{for } k\in \R.
\end{align*}
Again, in case of $k=0$ we will write $\mathbf{L}^2$ instead of $\mathbf{H}^0$ and we will neglect the subscript in the notation for the norm and the scalar product.
 Sometimes, on the space $\mathbf{H}^k$ we will consider the norms 
 \begin{equation}
     \vertiii{ \cdot}_{\mathbf{H}^k}^2= h_1 \lVert \cdot_1 \rVert_{H^k}^2+h_2 \lVert \cdot_2 \rVert_{H^k}^2 
 \end{equation}
 which are straightforwardly equivalent to the standard ones that we denote by $\lVert \cdot \rVert_{\mathbf{H}^k}$. Furthermore we define the following functional spaces 
\begin{align*}
    \cH=\mathbf{L}^2  ,\ \cV=\mathbf{H}^1,\ D\left( \mathbf{\Delta}\right)  =\mathbf{H}^2
\end{align*} 
where $\mathbf{\Delta}:D\left(  \mathbf{\Delta}\right)
\subset \cH\rightarrow \cH$ is defined by
\begin{align*}
    \mathbf{\Delta}(q_1,q_2)^t=(\Delta q_1,\Delta q_2)^t.
\end{align*}
 It is well known that $\mathbf{\Delta}$ is the infinitesimal generator of analytic semigroup of negative type and moreover $\cV$ can be identified with $D((-\mathbf{\Delta})^{1/2})$, see e.g. \cite{pazy2012semigroups}. 


A first issue related to the analysis of the dissipation properties of the transport noise is the well-posedness of such system. In fact the existence of strong probabilistic solution is outside the framework treated in earlier works, see for example \cite{thesis, cotter2020data, cotter2020modelling, Holm_2015}.
Thus, first, we need to define our notion of solution for system (\ref{eq:QG_ito}) and state our well-posedness result.

Let $Z$ be a separable Hilbert space, with associated norm $\| \cdot\|_{Z}$. For $p\geq 1$ denote by $L^p(\mathcal{F}_{t_0},Z)$ the space of $p$-integrable random variables with values in $Z$, measurable with respect to $\mathcal{F}_{t_0}$. Moreover, denote by $C_{\mathcal{F}}\left(  \left[  0,T\right]  ;Z\right) $ the space of continuous adapted processes $\left(  X_{t}\right)  _{t\in\left[
0,T\right]  }$ with values in $Z$ such that
\[
\E \left[  \sup_{t\in\left[  0,T\right]  }\left\Vert X_{t}\right\Vert
_{Z}^{2}\right]  <\infty
\]
and by $L_{\mathcal{F}}^{p}\left(  0,T;Z\right)$ the space of progressively
measurable processes $\left(  X_{t}\right)  _{t\in\left[  0,T\right]  }$ with
values in $Z$ such that
\[
    \E \left[  \int_{0}^{T}\left\Vert X_{t}\right\Vert _{Z}^{p}dt\right]
<\infty.
\]
Define for $j\in \{1,2\}$ the operators
\begin{align}
    F_{j}(q)&:=\frac{1}{2}\sum_{k=1}^K\bm{\sigma}_{j,k} \cdot \nabla  (\bm{\sigma}_{j,k}\cdot \nabla q), \label{eq:F_j(q)}\\
    G^k_{j}(q)&:=\bm{\sigma}_{j,k}\cdot \nabla q,\ \text{for } k\in \{1,\dots, K\}.\label{eq:Gkj(q)}
\end{align}
which can be easily shown to be bounded linear operators
$$ 
F_j\in \mathcal{L}(H^2(\mathcal{D}),L^2(\mathcal{D})), \quad G^k_j\in \mathcal{L}(H^{1}\left(\mathcal{D}\right);L^2(\mathcal{D})).$$ 
We then consider the following concept of weak solution (see e.g. \cite[Chapter~7]{DaPratoZab2014}) for \eqref{eq:QG_ito}:
\begin{definition}\label{weak solution qg}
    A stochastic process
    \[
    \mathbf{q}\in C_{\mathcal{F}}\left(  \left[  0,T\right] ;\cH\right)  \cap
    L_{\mathcal{F}}^{2}\left(  0,T; \cV \right)
    \]
    is a weak solution of equation \eqref{eq:QG_ito} if, for every $\bm{\phi}=(\phi_1,\phi_2)^t\in D(\mathbf{\Delta})  $, we have
    \begin{align*}
    \left\langle q_1(t)  ,\phi_1\right\rangle  &  =\left\langle
    q_{1,0},\phi_1\right\rangle +\int_{0}^{t}\left\langle \Delta \psi_1(s)
    , \nu\Delta  \phi_1\right\rangle ds+\int_0^t \langle q_1(s), F_1 \phi_1\rangle\ ds+\int_0^t \langle q_1(s), \nabla^{\perp}\psi_1(s)\cdot\nabla \phi_1\rangle ds \\
    &  +\int_{0}^{t}\left\langle \beta \partial_x\psi_1\left (s\right)  ,\phi_1\right\rangle
    ds +\int_{0}^{t}\left\langle F(s)  ,\phi_1\right\rangle
    ds-\sum_{k=1}^K\int_{0}^{t}\left\langle q_1(s)  ,G^k_1\phi_1\right\rangle dW_{s}^{1,k}\\ 
    \left\langle q_2(t)  ,\phi_2\right\rangle  &  =\left\langle
    q_{2,0},\phi_2\right\rangle +\int_{0}^{t}\left\langle \Delta \psi_2(s)
    , \nu\Delta  \phi_2\right\rangle ds+\int_0^t \langle q_2(s), F_2 \phi_2\rangle\ ds+\int_0^t \langle q_2(s), \nabla^{\perp}\psi_2(s)\cdot\nabla \phi_2\rangle ds \\
    &  +\int_{0}^{t}\left\langle \beta \partial_x\psi_2\left (s\right)  ,\phi_2\right\rangle
    ds -r\int_{0}^{t}\left\langle \Delta \psi_2(s)  ,\phi_2\right\rangle
    ds-\sum_{k=1}^K\int_{0}^{t}\left\langle q_2(s)  ,G^k_2\phi_2\right\rangle dW_{s}^{2,k}%
    \end{align*}
    for every $t\in\left[  0,T\right]  $, $\mathbb{P}$-a.s., where $\mathbf{q}$ and $\bfpsi$ are linked by relation \eqref{eq:relation_q_psi}.
\end{definition}
The well posedness of this system is guaranteed by the following result which we will prove in \autoref{sec well posed}.
\begin{theorem}\label{thm well psd qg vorticity}
    For every $\mathbf{q}_{0}\in L^4_{\mathcal{F}_0}(\Omega, \cH)$ and $F\in L^{4}_{\mathcal{F}}\left(  0,T;L^2(\mathcal{D})\right)  $, there exists one and only one weak solution of equation (\ref{eq:QG_ito}).
\end{theorem}
\begin{remark}
    We stated \autoref{thm well psd qg vorticity} in full generality in order to provide a complete framework for the well posedness of such stochastic system, that, to the best of these authors knowledge, is unavailable in the literature. This result is redundant for the scope of this work though. In fact, in order to exploit the eddy dissipation properties of the transport noise we will consider deterministic initial conditions and a deterministic and time-independent forcing.
\end{remark}

\begin{remark}
It is well known that, in absence of external forcing or dissipation, the quasi-geostrophic model has an infinite number of preserved quantities. The noise introduced in \eqref{eq:QG_strato} will preserve the $L^2$-norm of the quasi-geostrophic potential vorticity $\q$, namely the potential enstrophy of the system. For more on the role of enstrophy in the two-layer quasi-geostrophic dynamics see e.g. \cite{Lucarinietal2014}, \cite[Section~5.6.3, Section~9.2.2]{Vallis06}.
\end{remark}
After showing the well posedness of system (\ref{eq:QG_ito}), our goal is to provide sufficient conditions in order to model the dissipation properties of the transport noise. Let us explain briefly the adopted strategy before going into the details. 
First, following the approach introduced in \cite{flandoli2021quantitative}, we will show that under an appropriate scaling of the noise, first introduced by Galeati in \cite{galeati2020convergence}, the solution of the stochastic system \eqref{eq:QG_strato} converges, in a suitable sense which we will clarify later, to the solution of the associated deterministic system with an extra diffusion. 

Second, for time-independent forcings $F$ and same properties of the noise, solutions of the stochastic system \eqref{eq:QG_strato} will converge to those of the associated stationary deterministic 2LQG model.

In order to precisely formulate these statements, let us introduce some notations and the precise formulation of the noise. Let $a,b$ be two positive numbers, then we write $a\lesssim b$ if there exists a positive constant $C$ such that $a \leq C b$ and $a\lesssim_\alpha b$ when we want to highlight the dependence of the constant $C$ on a parameter $\alpha$.
Let $e_k(\mathbf{x})=L^{-1}\exp{\frac{2\pi i}{L} k\cdot \x}$, $ k\in \mathbb{Z}^2_0=\mathbb{Z}^2\setminus\{(0,0)\}$, orthonormal basis of $L^2(\mathcal{D})$ made by eigenfunctions of $-\Delta$. 

Following the approach first introduced in \cite{galeati2020convergence} consider the following explicit representations of the coefficients $\bm{\sigma}_{j,k},\ k\in \mathbb{Z}^2_0,\ j\in \{1,2\}$
\begin{equation}
    \bm{\sigma}_{j,k}(\x)=\sqrt{2\kappa}\bm{a}_{j,k} e_k(\x)=
    \begin{cases}\sqrt{2\kappa}\theta_{j,k} e_k(\x)\frac{k^{\perp}}{\lvert k\rvert}\ \textit{ if } k\in \mathbb{Z}_+^2\\
    \sqrt{2\kappa}\theta_{j,k} e_k(\x)\frac{-k^{\perp}}{\lvert k\rvert}\ \textit{ if }  k\in \mathbb{Z}_-^2,
    \end{cases} 
\end{equation}
where $\mathbb{Z}^2_+,\ \mathbb{Z}^2_-$ is a partition of $\mathbb{Z}^2_0$ with $\mathbb{Z}^2_+=-\mathbb{Z}^2_-$, and the parameters $\theta_{j,k}$ satisfy the following conditions:
\begin{enumerate}
    \item $\sum_{k\in \mathbb{Z}^2_0}\theta_{j,k}^2=1 $.
    \item $\theta_{j,k}=0$ if $\lvert k\rvert$ is large enough. We will denote by $K$ the finite set of $k$ where $\theta_{j,k}\neq 0$.
    \item $\theta_{j,k}=\theta_{j,l}$ if $\lvert k\rvert=\lvert l\rvert$.
\end{enumerate}
Furthermore take an infinite sequence of complex standard Brownian motions such that $\overline{W^{j,k}}=W^{j,-k}$. Thus the noise we consider is parameterized by the coefficients $\kappa,\ \theta_{j,k}$ and the set $K$. Under this setting, as described for example in \cite{galeati2020convergence, dong2022dissipation}, equation \eqref{eq:QG_ito} can be reformulated as 
\begin{align}
\label{eq:QG_ito GL}
\begin{split}
    dq_1 + \left(\nabla^{\perp}\psi_1\cdot \nabla q_1\right) dt = \left(\kappa\Delta q_1+\nu\Delta^2 \psi_1 - \beta \partial_x \psi_1 + F \right)\, dt + \sqrt{2\kappa}\sum_{k=1}^K\bm{a}_{1,k}e_k\cdot \nabla q_1\,  dW^{1,k}\\
    dq_2 + \left(\nabla^{\perp}\psi_2\cdot \nabla q_2\right) dt = \left(\kappa\Delta q_2 + \nu\Delta^2 \psi_2 -\beta \partial_x \psi_2 - r\Delta \psi_2\right)\, dt + \sqrt{2\kappa}\sum_{k=1}^K\bm{a}_{2,k}e_k\cdot \nabla q_2 \, dW^{2,k} .
    \end{split}
\end{align}
The corresponding deterministic system is 
\begin{align}
\label{eq:QG_deterministic GL}
\begin{split}
    d\Bar{q}_1 + \left(\nabla^{\perp}\Bar{\psi}_1\cdot \nabla \Bar{q}_1\right) dt = \left(\kappa\Delta \Bar{q}_1+\nu\Delta^2 \Bar{\psi}_1 - \beta \partial_x \Bar{\psi}_1 + F \right)\, dt\\
    d\Bar{q}_2 + \left(\nabla^{\perp}\Bar{\psi}_2\cdot \nabla \Bar{q}_2\right) dt = \left(\kappa\Delta \Bar{q}_2 + \nu\Delta^2 \Bar{\psi}_2 -\beta \partial_x \Bar{\psi}_2 - r\Delta \Bar{\psi}_2\right)\, dt,
    \end{split}
\end{align}
where, as before, $\Bar{\mathbf{q}}$ and $\Bar{\bfpsi}$ are linked by relation (\ref{eq:relation_q_psi}).

Thanks to Theorem~\ref{thm well psd qg vorticity} and classical results on two dimensional deterministic quasi-geostrophic equations, see for example \cite{Bernier94}, under the assumptions $\q_0\in \cH,\ F\in L^4(0,T;L^2(\mathcal{D}))$ there exists a unique weak solutions $\q$ (resp. $\Bar{\q}$) of the problem \eqref{eq:QG_ito GL} (resp. \eqref{eq:QG_deterministic GL}).
Now we can state one of the main results of our work which allows to quantify the difference between the behavior of the stochastic and the deterministic system.
\begin{theorem}\label{thm galeati}
Let $\q$ and $\bar{\q}$ be weak solutions to (\ref{eq:QG_ito GL}) and (\ref{eq:QG_deterministic GL}) respectively. Then for any $\alpha\in (0,1),$ there exists $C$ depending from $\alpha$ and all the parameters of the model except for the noise such that for any $\epsilon\in (0,\alpha]$ one has
\begin{itemize}
    \item[(i)] \begin{align*}
    \E \left[\lVert \q-\bar\q\rVert_{C([0,T];\mathbf{H}^{-\alpha})}^2\right]\lesssim_{\alpha,M,\epsilon,T}& \kappa^{\epsilon}\lVert \theta\rVert_{\ell^{\infty}}^{2(\alpha-\epsilon)}R_T^2 \exp\left(T\frac{\nu^2+\beta^2+r^2}{\kappa+\nu}\right)\\ & \exp\left(\frac{CTR_T^2}{(\kappa+\nu)^2}\left(1+\kappa+\nu\right)+\frac{C}{\left(\kappa+\nu\right)^2}\int_0^T\lVert F(s)\rVert^2\ ds\right).
\end{align*}
\item[(ii)]  \begin{align*}
    \E \left[\lVert \q-\bar\q\rVert_{C([0,T];\mathbf{H}^{-\alpha})}^2\right]\lesssim_{\alpha,M,\epsilon,T}& \kappa^{\epsilon}\lVert \theta\rVert_{\ell^{\infty}}^{2(\alpha-\epsilon)}R_T^2\exp\left(T\frac{\nu^2+\beta^2+r^2}{\kappa+\nu}\right)\\ &\exp\left(\frac{C}{\nu^2}\left(TR_T^2+\int_0^T\lVert F(s)\rVert^2\ ds \right)\right).
\end{align*} 
\end{itemize}
where $R_T^2$ is a constant independent of the noise defined in Section \ref{subsec:convergence}.
\end{theorem}
%

Next, given the model \eqref{eq:QG_strato} now with a time-independent forcing $F$, consider the associated stationary system. On the one hand, similarly to the Navier-Stokes system, existence and uniqueness of the solution of the stationary system associated to quasi-geostrophic equations is guaranteed under, generally unfeasible, assumptions on the viscosity. On the other hand, as we will show in \autoref{sec long time}, in our framework these assumptions will be satisfied thanks to the dissipation properties of the transport noise. 
More precisely, consider $\tilde{\q}$ solution of the following system
\begin{align}
\label{eq:QG_stationary}
\begin{split}
    \nabla^{\perp}\tilde{\psi}_1\cdot \nabla \tilde{q}_1 = \kappa\Delta \tilde{q}_1+\nu\Delta^2 \tilde{\psi}_1 - \beta \partial_x \tilde{\psi}_1 + F\\
    \nabla^{\perp}\tilde{\psi}_2\cdot \nabla \tilde{q}_2 = \kappa\Delta \tilde{q}_2+\nu\Delta^2 \tilde{\psi}_2 -\beta \partial_x \tilde{\psi}_2 - r\Delta \tilde{\psi}_2.
    \end{split}
\end{align}
where, $F\in L^2(\mathcal{D})$ and as always, $\tilde{\mathbf{q}}$ and $\tilde{\bfpsi}$ are linked by relation (\ref{eq:relation_q_psi}). Then, thanks to a particular parametrization of the noise, we will show in \autoref{sec long time} the following:
\begin{theorem}\label{thm:convergence deterministic}
For $\kappa$ large enough, for each $\delta>0$ and $\alpha\in (0,1)$, it exists $\Bar{T}=\Bar{T}(\delta)$ and a sequence $\{\theta_{j,k}\}_{k\in K,j\in \{1,2\}}$ depending from $\delta,\ \Bar{T}$, $\alpha$ such that for each $t\in [\Bar{T},2\Bar{T}]$
\begin{align*}
    \E \left[\lVert \q(t)-\tilde{\q}\rVert^2_{\mathbf{H}^{-\alpha}}\right]\leq \delta.
\end{align*}
\end{theorem}
%


\section{Preliminaries}
\label{preliminary sec}
In this section we recall several technical tools crucial for the next sections, see for example \cite{breckner1999approximation},\cite{carigi2022exponential},\cite{flandoli2021quantitative},\cite{pazy2012semigroups} for more details.

First, similarly to \cite{carigi2022exponential}, we define the linear operator $\tilde{A}:\mathbf{H}^{k+2}\rightarrow \mathbf{H}^k,\ k\in \mathbb{R}$ connecting the streamfunction with the quasi–geostrophic potential vorticity
\begin{align*}
    \tilde{A}\bm{\psi}:=\left(-\bm{\Delta}-M\right)\bm{\psi}, \textit{ where } M= \begin{pmatrix}
-S_1& S_1 \\
S_2 & -S_2.
\end{pmatrix}
\end{align*}
It is well known that it has a bounded inverse $\left(-\bm{\Delta}-M\right)^{-1}:\mathbf{H}^k\rightarrow \mathbf{H}^{k+2}$.
Thanks to this fact, for each $\q\in \mathbf{H}^k$, there exists a unique $\bfpsi\in \mathbf{H}^{k+2}$ such that $\q=\left(\mathbf{\Delta}+M\right)\bfpsi $ and moreover for each $k\in \mathbb{R}$ there exists two constants $c_{1,k}\leq c_{2,k}$ such that
\begin{align*}
    c_{1,k}\lVert \bfpsi\rVert_{\mathbf{H}^{k+2}}\leq \lVert \q\rVert_{\mathbf{H}^{k}}\leq c_{2,k}\lVert \bfpsi\rVert_{\mathbf{H}^{k+2}}.
\end{align*}
In this work we will also use extensively the relation
\begin{equation}
\label{eq:difference relation}
    \psi_1-\psi_2=(-\Delta+S_1+S_2)^{-1}(q_2-q_1)
\end{equation}
which follows directly from \eqref{eq:relation_q_psi}.


We recall three technical lemmata which can be proved by classical arguments and we refer to \cite[Section~2.1]{flandoli2021quantitative}. 
\begin{lemma}[{\cite[Lemma~2.1]{flandoli2021quantitative}}]\label{preliminary non linear}
Given a divergence free vector field $ V\in L^2(\mathcal{D};\mathbb{R}^2)$ the following bounds hold true.
\begin{enumerate}
    \item If $ V\in L^{\infty}\left(\mathcal{D};\mathbb{R}^2\right)$, $f\in L^2(\mathcal{D})$, then we have \begin{align*}
        \lVert V\cdot \nabla f\rVert_{H^{-1}}\lesssim \lVert V\rVert_{L^{\infty}}\lVert f\rVert.
    \end{align*}
    \item Let $\alpha\in (1,2],\ \beta\in (0,\alpha-1),\ V\in H^{\alpha}(\mathcal{D};\mathbb{R}^2),\ f\in {H}^{-\beta}(\mathcal{D}),$ we have
    \begin{align*}
        \lVert V\cdot \nabla f\rVert_{H^{-1-\beta}}\lesssim_{\alpha,\beta} \lVert V\rVert_{H^{\alpha}}\lVert f\rVert_{{H}^{-\beta}}.
    \end{align*}
    \item Let $\beta \in (0,1),$ then for any $f\in H^{\beta}(\mathcal{D}),\ g\in H^{1-\beta}(\mathcal{D})$ it holds
    \begin{align*}
        \lVert fg\rVert\lesssim_{\beta } \lVert f\rVert_{H^{\beta}}\lVert g\rVert_{H^{1-\beta}}
    \end{align*}
    \item Let $\beta\in (0,1),\ V\in H^{1-\beta}(\mathcal{D};\mathbb{R}^2),\ f\in L^2(\mathcal{D}),$ then one has
    \begin{align*}
        \lVert V\cdot \nabla f\rVert_{H^{-1-\beta}}\lesssim_{\beta} \lVert V\rVert_{H^{1-\beta}}\lVert f\rVert.
    \end{align*}
\end{enumerate}
\end{lemma}
\begin{remark}
With some abuse of notations, if $\bfpsi$ and $\q$ are two vector fields, we will denote  $$\nabla^{\perp}\bfpsi\cdot \nabla\q:=\left(\nabla^{\perp}\psi_1\cdot\nabla q_1,\nabla^{\perp}\psi_2\cdot\nabla q_2\right)^t$$
and the results of previous lemma continue to hold in this framework. Then this first lemma generalises fundamental estimates for the nonlinearity of the quasi-geostrophic model presented for example in \cite{carigi2022exponential, Chueshov01_Proba}.
\end{remark}
The second lemma provides classical estimates on the semigroup generated by $\bm{\Delta}$:
\begin{lemma}[{\cite[Lemma~2.2]{flandoli2021quantitative}}]\label{lemma:12_2.2}
Let $\q\in \mathbf{H}^{\alpha},\ \alpha\in \mathbb{R}$. Then:
\begin{enumerate}
    \item for any $\rho\geq 0,$ it holds $\lVert e^{t\mathbf{\Delta}}\q\rVert_{\mathbf{H}^{\alpha+\rho}}\leq C_{\rho}t^{-\rho/2}\lVert \q\rVert_{\mathbf{H}^{\alpha}}$ for some constant increasing in $\rho$;
    \item for any $\rho\in [0,2],$ it holds $\lVert \left( I-e^{t\mathbf{\Delta}}\right)\q\rVert_{\mathbf{H}^{\alpha-\rho}}\lesssim_{\rho} t^{\rho/2}\lVert \q\rVert_{\mathbf{H}^{\alpha}}$.
\end{enumerate}
\end{lemma}
The semigroup $e^{\delta(t-s)\mathbf{\Delta}}$ has also regularising effects as stated in the following:
\begin{lemma}[{\cite[Lemma~2.3]{flandoli2021quantitative}}]\label{lemma:12_2.3}
For any $\delta>0,\ \alpha\in \mathbb{R},\ \q\in L^2(0,T;\mathbf{H}^{\alpha})$, it holds
\begin{align*}
    \left\lVert \int_0^t e^{\delta(t-s)\mathbf{\Delta}}\q(s) \ ds\right\rVert_{\mathbf{H}^{\alpha+1}}^2\lesssim\frac{1}{\delta}\int_0^t \lVert \q(s)\rVert_{\mathbf{H}^{\alpha}}^2\ ds\ \ \forall t\in [0,T].
\end{align*}
\end{lemma}
Finally consider the following classical result.
\begin{lemma}[{\cite[Proposition~B.3]{breckner1999approximation}}] \label{lemma:convergence}
If $\{Q_N\}_{N\geq 1}\subseteq L^2(\Omega,\mathcal{F},\mathbb{P};L^2(0,T;\mathbb{R}))$ are continuous stochastic processes, $\{\sigma_M\}_{M\geq 1}$ are $\mathcal{F}_t$-stopping times increasing to $T$ and such that \begin{align*}
    \operatorname{sup}_{N\geq 1}\E \left[\lvert Q_N(T)\rvert^2\right]<+\infty\\ \lim_{N\rightarrow+\infty}\E \left[\lvert Q_N(\sigma_M)\rvert\right]=0,\ \ \forall M\geq 1
\end{align*}
then $\E \left[\lvert Q_N(T)\rvert\right]\rightarrow 0$.
\end{lemma}

\section{Well-posedness}\label{sec well posed}
In this section we will show \autoref{thm well psd qg vorticity} holds following a classical approach by means of the Galerkin approximation.
\subsection{Galerkin Approximation and Limit Equations} \label{Sec Galerkin approximation}
As described in Section \ref{sec model}, let $\{e_i\}_{i\in \mathbb{N}}$ be the orthonormal basis of $L^2(\mathcal{D})$ made by eigenvectors of $-\Delta$ and $\lambda_i$ the corresponding eigenvalues, $\lambda_i$ are positive and non decreasing. 
Given $\mathcal{H}^N=\operatorname{span} \{e_1,\dots,\ e_N\}\subseteq L^2(\mathcal{D})$ let $P^N:L^2(\mathcal{D})\rightarrow L^2(\mathcal{D})$ be the orthogonal projector of $L^2(\mathcal{D})$ on $\mathcal{H}^N.$  
As we are looking for a finite dimensional approximation of the solution of equation (\ref{eq:QG_ito}), define
$$ q_j^N(t):=\sum_{i=1}^N c_{i,j,N}(t)e_i(x).$$ 
The coefficients $c_{i,j,N}$ are chosen in such a way to satisfy for all eigenfunctions $e_i,\ \ 1\leq i\leq N$ and all $t\in [0,T]$ the system
\begin{align}
    \label{weak formulation Galerkin vorticity}
    \begin{split}
    \left\langle q_1^N(t)  ,e_i\right\rangle  &  =\left\langle
q_{1,0}^N,e_i\right\rangle +\int_{0}^{t}\left\langle \Delta \psi_1^N(s)
, \nu\Delta  e_i\right\rangle ds+\int_0^t \langle q_1^N(s), F_1^N e_i\rangle\ ds+\int_0^t \langle q_1^N(s), \nabla^{\perp}\psi_1^N(s)\cdot\nabla e_i\rangle ds \\
&  +\int_{0}^{t}\left\langle \beta \partial_x\psi_1^N\left (s\right)  ,e_i\right\rangle
ds +\int_{0}^{t}\left\langle F(s)  ,e_i\right\rangle
ds-\sum_{k=1}^K\int_{0}^{t}\left\langle q_1^N(s)  ,G^{k}_1 e_i\right\rangle dW_{s}^{1,k}\\ 
\left\langle q_2^N(t)  ,e_i\right\rangle  &  =\left\langle
q_{2,0}^N,e_i\right\rangle +\int_{0}^{t}\left\langle \Delta \psi_2^N(s)
, \nu\Delta  e_i\right\rangle ds+\int_0^t \langle q_2^N(s), F_2^N e_i\rangle\ ds+\int_0^t \langle q_2^N(s), \nabla^{\perp}\psi_2(s)\cdot\nabla e_i\rangle ds \\
&  +\int_{0}^{t}\left\langle \beta \partial_x\psi_2^N\left (s\right)  ,e_i\right\rangle
ds -r\int_{0}^{t}\left\langle \Delta \psi_2^N(s)  ,e_i\right\rangle
ds-\sum_{k=1}^K\int_{0}^{t}\left\langle q_2^N(s)  ,G^{k}_2 e_i\right\rangle dW_{s}^{2,k} .
    \end{split}
\end{align}
Here $\q^N_0=P^N \q_0$ and the variables $\bfpsi^N$ and $\q^N$ are linked by relation (\ref{eq:relation_q_psi}), and the operators $F_j^N$, $i = 1,2$, are defined similarly to \eqref{eq:F_j(q)}, namely 
\begin{equation*}
    F_j^N\phi=\frac{1}{2}\sum_{k=1}^K P^N\left(\bm{\sigma}_{j,k} \cdot \nabla  P^N(\bm{\sigma}_{j,k}\cdot \nabla \phi)\right),\ j\in \{1,2\} \ \fa \phi \in \mathcal{H}^N.
\end{equation*}
The local well-posedness of \eqref{weak formulation Galerkin vorticity} follows from classical results about stochastic differential equations with locally Lipshitz coefficients, see for example \cite{karatzas2012brownian},\cite{skorokhod1982studies}. The global well-posedness follows from the following a priori estimates:
\begin{lemma}
\label{stime energia} 
    Given the system \eqref{weak formulation Galerkin vorticity} the following energy estimate holds:
    \begin{align}\label{ITO}
        \frac{d \vertiii{\q^N} ^2}{2}+\nu \vertiii{\nabla \q^N}^2 dt& =\left(-\beta h_1\langle \partial_x \psi_1^N,q_1^N\rangle-\beta h_2\langle \partial_x \psi_2^N,q_2^N\rangle+h_1\langle F,q_1^N\rangle-rh_2 \lVert q_2^N\rVert^2+Sr\langle \psi_1^N-\psi_2^N,q_2^N\rangle\right) dt\notag\\ & +\left(S\nu\lVert q_1^N-q_2^N\rVert^2+S\nu(S_1+S_2)\langle \psi_1^N-\psi_2^N,q_1^N-q_2^N\rangle\right)\, dt.
    \end{align}
    Furthermore for any $\q_0\in L_{\cF_0}^4(\Omega, \cH)$ and $F\in L_{\cF}^4(0, T; L^2(\cD))$ the following \textit{a priori} and integral bounds are satisfied uniformly in $N\in \N$
    \begin{align}
        \E \left[\sup_{t\in [0,T]}\vertiii{\mathbf{q}^N(t)} ^2\right]\leq  \E \left[\vertiii{\mathbf{q}_0}^2 +2\int_0^T \lVert F(s)\rVert^2 ds\right]  e^{CT} \label{aprioriL2}\\
         \nu\E \left[\int_0^T \vertiii{\nabla \mathbf{q}^N(s)}^2\ ds\right]\leq \E \left[\int_0^T \lVert F(s)\rVert^2 \ ds\right]+CT  \E \left[\vertiii{\mathbf{q}_0}^2 +2\int_0^T \lVert F\rVert^2 ds\right]  e^{CT} \label{integralL2}\\
       \mathbb{E}\left[\operatorname{sup}_{t\in [0,T]}\vertiii{\mathbf{q}^N(t)} ^4\right]+\mathbb{E}\left[\int_0^T \vertiii{\mathbf{q}^N(s)}^2\vertiii{\nabla \mathbf{q}^N(s)}^2\ ds \right]\leq C(T) \label{energy L4} \\ 
        \E \left[\left(\int_0^T \vertiii{\nabla\mathbf{q}^N(s)}^2\ ds \right)^2\right]\leq C(T).\label{energy strana}
    \end{align}
    where $C$ is a constant possibly changing its value line by line, but independent of $N$.
    
\end{lemma}
\begin{proof}
Let us start by showing \eqref{ITO}. We first apply the finite dimensional It\^{o}'s formula to the system (\ref{weak formulation Galerkin vorticity}) to get
\begin{align*}
    \frac{d\lVert q_1^N\rVert^2}{2}=\left( \nu \langle \Delta\psi_1^N,\Delta q_1^N\rangle - \beta\langle \partial_x \psi_1^N,q_1^N\rangle+\langle F,q_1^N\rangle\right)dt\\
    \frac{d\lVert q_2^N\rVert^2}{2}=\left( \nu \langle \Delta\psi_2^N,\Delta q_2^N\rangle-\beta\langle \partial_x \psi_2^N,q_2^N\rangle-r\langle \Delta \psi_2^N,q_2^N\rangle\right)dt,
\end{align*}
 then multiply each equation by $h_1$ and $h_2$ respectively and sum them up to obtain
\begin{multline} \label{eq:energyIto1}
    \frac{d \vertiii{\mathbf{\mathbf{q}^N}}^2}{2} = \left(h_1\nu \langle \Delta\psi_1^N,\Delta q_1^N\rangle +h_2 \nu \langle \Delta\psi_2^N,\Delta q_2^N\rangle\right)\, dt + \left(-\beta h_1\langle \partial_x \psi_1^N,q_1^N\rangle -\beta h_2\langle \partial_x \psi_2^N, q_2^N \rangle  \right)\, dt \\
    \left( + h_1\langle F,q_1^N\rangle - rh_2 \langle \Delta\psi_2^N, q_2^N \rangle \right)\, dt.
\end{multline}
Now, thanks to relations (\ref{eq:relation_q_psi}), \eqref{eq:defS} and \eqref{eq:difference relation} we have
\begin{align*}
    h_1\langle \Delta\psi_1^N,\Delta q_1^N\rangle + h_2 \nu \langle \Delta\psi_2^N,\Delta q_2^N\rangle &= h_1\langle q_1^N + S_1(\psi_1  -\psi_2),\Delta q_1^N\rangle + h_2\langle q_2^N + S_2(\psi_2  -\psi_1),\Delta q_2^N \rangle\\
    &= - \vertiii{ \nabla \q^N}^2 + S\langle \Delta\psi_1^N - \Delta\psi_2^N, q_1^N - q_2^N\rangle \\
    &= - \vertiii{ \nabla \q^N}^2 + S\| q^N_1 - q^N_2\|^2 + S(S_1 + S_2)\langle \psi_1^N - \psi_2^N , q_1^N - q_2^N\rangle .
\end{align*}
Using \eqref{eq:relation_q_psi} to treat similarly the term $rh_2 \langle \Delta\psi_2^N, q_2^N \rangle $, from \eqref{eq:energyIto1} we get the desired result
\begin{align*}
    \frac{d \vertiii{\q^N}^2}{2}+\nu\vertiii{\nabla \q^N}^2 dt & =\left(-\beta h_1\langle \partial_x \psi_1^N,q_1^N\rangle-\beta h_2\langle \partial_x \psi_2^N,q_2^N\rangle+h_1\langle F,q_1^N\rangle-rh_2 \lVert q_2^N\rVert^2+Sr\langle \psi_1^N-\psi_2^N,q_2^N\rangle\right) dt\\ & +\left(S\nu\lVert q_1^N-q_2^N\rVert^2+\nu S(S_1+S_2)\langle \psi_1^N-\psi_2^N,q_1^N-q_2^N\rangle\right)dt.
\end{align*}
Now, by Cauchy-Schwartz and exploiting the continuity of the operator 
 \begin{equation*}
      (-\mathbf{\Delta}-M)^{-1}: \mathbf{H}^s\rightarrow  \mathbf{H}^{s+2}, \text{ where } M= \begin{pmatrix}
-S_1& S_1 \\
S_2 & -S_2
\end{pmatrix},
 \end{equation*}
     it follows that there exists a constant $C$ independent of $N$ such that 
     \begin{equation}
         \label{eq:sec4 galerkin 1}
         \frac{d \vertiii{\mathbf{q}^N}^2}{2}+\nu\vertiii{\nabla \mathbf{q}^N}^2 dt\leq\left( C \vertiii{\mathbf{q}^N}^2+ \lVert F\rVert^2\right) dt.
     \end{equation}
    Thus, via Gronwall's inequality and exploiting the fact that $P^N$ are projections on $L^2(\mathcal{D})$, we have that 
    \begin{align*}
        \vertiii{\mathbf{q}^N(t)}^2 \leq \left(\vertiii{\mathbf{q}_0}^2 +2\int_0^t \lVert F(s)\rVert^2 ds \right) e^{Ct},
    \end{align*}
    hence \eqref{aprioriL2} holds taking the supremum over $[0,T]$ and the expectation on both sides. Next, integrating in time \eqref{eq:sec4 galerkin 1} and using the estimate just obtained we have
    \begin{equation*}
        \nu\int_0^T \vertiii{\nabla \mathbf{q}^N(s)}^2\ ds\leq \int_0^T \lVert F(s)\rVert^2 \ ds+ CT \left( \vertiii{\mathbf{q}_0}^2 +2\int_0^T \lVert F(s)\rVert^2 ds \right)e^{CT} .
    \end{equation*}
    from which follows \eqref{integralL2}.
    
    %
    Moving on to \eqref{energy L4}, since $d\vertiii{\mathbf{q}^N} ^4=2\vertiii{\mathbf{q}^N}^2d\vertiii{\mathbf{q}^N}^2$, exploiting the energy estimate (\ref{ITO}) we obtain
    \begin{multline}\label{eq:proof energyL4 1}
        d\vertiii{\mathbf{q}^N}^4+4\nu \vertiii{\mathbf{q}^N} ^2 \vertiii{\nabla \mathbf{q}^N} ^2 \, dt = 4\vertiii{\mathbf{q}^N}^2\left(-\beta h_1\langle \partial_x \psi_1^N,q_1^N\rangle-\beta h_2\langle \partial_x \psi_2^N,q_2^N\rangle+h_1\langle F,q_1^N\rangle\right)\, dt +\\
        4\vertiii{\mathbf{q}^N}^2\left(-rh_2 \lVert q_2^N\rVert^2+Sr\langle \psi_1^N-\psi_2^N,q_2^N\rangle\right) \, dt  +4\vertiii{\mathbf{q}^N}^2\left(S\nu\lVert q_1^N-q_2^N\rVert^2+S\nu(S_1+S_2)\langle \psi_1^N-\psi_2^N,q_1^N-q_2^N\rangle\right)\, dt.
    \end{multline}
    As argued for \eqref{aprioriL2}, by appropriate use of Cauchy-Schwartz inequality and of the continuity of the operator $(- \Delta - M)^{-1}$, we can estimate the right hand side of \eqref{eq:proof energyL4 1} as follows
    \begin{align}
         d\vertiii{\mathbf{q}^N}^4+4\nu \vertiii{\mathbf{q}^N} ^2 \vertiii{\nabla \mathbf{q}^N} ^2 \, dt  &\leq \left(C \vertiii{\mathbf{q}^N}^4+C\lVert F\rVert \vertiii{\mathbf{q}^N}^3\right)\, dt \nonumber\\
        & \leq \left(C \vertiii{\mathbf{q}^N}^4+\lVert F\rVert^4\right)\ dt. \label{eq:proof energyL4 2}
    \end{align}
    Thus, via Gronwall's inequality and the properties of the projection $P^N$, we get
    \begin{align*}
        \vertiii{\mathbf{q}^N(t)}^4 \leq \left(\vertiii{\mathbf{q}_0}^4+ \int_0^t \lVert F(s)\rVert^4 \ ds\right)e^{Ct}.
    \end{align*}
  From this relation it follows immediately that 
  \begin{align*}
        \E \left[\operatorname{sup}_{t\in [0,T]}\vertiii{\mathbf{q}^N(t)}^4\right]+\E \left[\int_0^T \vertiii{\mathbf{q}^N(s)}{^2} \vertiii{ \nabla\mathbf{q}^N(s)}^2\ ds \right]\leq C
    \end{align*}
    where $C= C(T)$ is a constant dependent on the time $T$ but crucially independent of $N$.
    
    Finally, to prove \eqref{energy strana} we integrate \eqref{eq:sec4 galerkin 1} over $[0,T]$ to get 
    \begin{equation*}
        \nu \int_0^T \vertiii{\nabla \q^N(s)}^2 \, ds \leq \tfrac{1}{2}\vertiii{\q_0}^2 + C\int_0^T \vertiii{\q^N(s)}^2 \, ds + \int_0^T \| F(s)\|^2\, ds.
    \end{equation*}
    Squaring and simply bounding the right hand side we have 
    \begin{align*}
        \nu^2\left(\int_0^T \vertiii{\nabla \mathbf{q}^N(s)}^2\ ds\right)^2\leq \tfrac{3}{2}\vertiii{\mathbf{q}_0} ^4+ 3\left(C T \sup_{t\in [0,T]} \vertiii{\mathbf{q}^N(s)} ^2 \right)^2+ 3\left(\int_0^T \lVert F(s)\rVert^2 \ ds\right)^2.
    \end{align*}
    By \eqref{aprioriL2} we have 
     \begin{align*}
        \nu^2\left(\int_0^T \vertiii{\nabla \mathbf{q}^N(s)}^2\ ds\right)^2\leq \tfrac{3}{2}\vertiii{\mathbf{q}_0} ^4+ 3\left[ C T e^{CT}\left( \vertiii{\q_0}^2 + 2\int_0^T\| F(s)\|^2\, ds \right) \right]^2+ 3\left(\int_0^T \lVert F(s)\rVert^2 \ ds\right)^2.
    \end{align*}
    Then taking the expectation on both side, \eqref{energy strana} holds with constant $C$ depending on $T$, $\mathbb{E}\left[\vertiii{\q_0}^4\right]$ and $\E\left[\|F\|^4_{L^2(0,T;\cH)}\right]$.  
    
\end{proof}
%
%
From the energy estimates for $\mathbf{q}^N$ shown in \autoref{stime energia}, it follows that there exists a subsequence, which we will denote again for simplicity by $\mathbf{q}^N$, converging to a stochastic process $\mathbf{q}$ $*$-weakly in $L^4(\Omega; L^{\infty}(0,T; \cH)) $, and weakly in  $L^4(\Omega; L^{2}(0,T; \cV))$. 
Furthermore there exist two unknown processes, $B_1^*,\ B_2^* $ such that 
\begin{equation}
\label{eq:defB12*}
    \nabla^{\perp}\psi_{j}^N\cdot\nabla q_j^{N}{\rightharpoonup} B^*_j \ \text{ in  } L^2(\Omega; L^{2}(0,T;H^{-1}(\mathcal{D})),\ \ j\in \{1,2\}.
\end{equation}
Moreover, thanks to the converging properties of the projector $P^N$ for $N\rightarrow+\infty$, the processes $\mathbf{q}$ and $B_j^*,\ j\in \{1,2\}$ satisfies $\bP$-a.s.\ for each $i\in \mathbb{N}$ and $t\in [0,T]$
\begin{align}\label{limit equation}
\begin{split}
    \left\langle q_1(t)  ,e_i\right\rangle+\int_0^t \langle B_1^*(s), e_i\rangle_{H^{-1},H^1} ds  &  =\left\langle
    q_{1,0},e_i\right\rangle +\int_{0}^{t}\left\langle \Delta \psi_1(s)
    , \nu\Delta  e_i\right\rangle ds+\int_0^t \langle q_1(s), F_1 e_i\rangle\ ds \\
    &  +\int_{0}^{t}\left\langle \beta \partial_x\psi_1\left (s\right)  ,e_i\right\rangle
    ds +\int_{0}^{t}\left\langle F(s)  ,e_i\right\rangle
    ds-\sum_{k=1}^K\int_{0}^{t}\left\langle q_1(s)  ,G^k_1 e_i\right\rangle dW_{s}^{1,k}\\ 
    \left\langle q_2(t)  ,e_i\right\rangle+\int_0^t \langle B_2^*(s), e_i\rangle_{H^{-1},H^1} ds  &  =\left\langle
    q_{2,0},e_i\right\rangle +\int_{0}^{t}\left\langle \Delta \psi_2(s)
    , \nu\Delta  e_i\right\rangle ds+\int_0^t \langle q_2(s), F_2 e_i\rangle\ ds \\
    &  +\int_{0}^{t}\left\langle \beta \partial_x\psi_2\left (s\right)  ,e_i\right\rangle
    ds -r\int_{0}^{t}\left\langle \Delta \psi_2(s)  ,e_i\right\rangle
    ds-\sum_{k=1}^K\int_{0}^{t}\left\langle q_2(s)  ,G^k_2 e_i\right\rangle dW_{s}^{2,k}.%
\end{split}
\end{align}
%
Let us expand on the convergence of the nonlinear term 
$$\int_0^t \langle q_j^N(s),F^N_j e_i\rangle \,ds.$$
We know that for each $\alpha\geq 0$ and $x\in D((-\Delta)^{\alpha}),\ \lVert P^N x - x\rVert_{D((-\Delta)^{\alpha})}$ vanishes on the limit $N\to \infty$. Thus, since for each $k\in K$, $\beta\geq 1/2$, the operator $\bm{\sigma}_{j,k}\cdot (\nabla \cdot)$ is in $L(D((-\Delta)^{\beta}),D((-\Delta)^{\beta-1/2}))$, then for $\phi\in D((-\Delta)^{\beta})$, $\lVert P^N(\bm{\sigma}_{j,k}\cdot \phi)- \bm{\sigma}_{j,k}\cdot \phi\rVert_{D((-\Delta)^{\beta-1/2})}$ converges to zero for increasing $N$. Starting from these observations it is easy to show that for each $i,\ \lVert F^N_j e_i-F_j e_i\rVert\rightarrow 0$. 
Then, thanks to the weak convergence of $q_j^N$ to $q_j$, we have the required convergence of $\int_0^t \langle q_j^N(s),F^N_j e_i\rangle ds$.

For what concern the continuity in $\cH$ of the process $\textbf{q}$ we can argue in the following way via It\^{o}'s formula and Kolmogorov continuity theorem. From the weak formulation above we get the weak continuity in $\cH$ of $\mathbf{q}$ applying the Kolmogorov continuity theorem for \eqref{limit equation}. Then, applying the It\^{o}'s formula to $\vertiii{\mathbf{q}(t)}^2$ we get, arguing as in the proof of Lemma~\ref{stime energia},
\begin{align*}
    \frac{d \vertiii{\mathbf{q}} ^2}{2}&=-\nu \vertiii{\nabla \mathbf{q}}^2 dt\left(-\beta h_1\langle \partial_x \psi_1,q_1\rangle-\beta h_2\langle \partial_x \psi_2,q_2\rangle+h_1\langle F,q_1\rangle-rh_2 \lVert q_2\rVert^2+Sr\langle \psi_1-\psi_2,q_2\rangle\right) dt\notag\\ & +\left(S\nu\lVert q_1-q_2\rVert^2+S\nu(S_1+S_2)\langle \psi_1-\psi_2,q_1-q_2\rangle\right)dt-\left(h_1\langle B^*_1,q_1\rangle+h_2\langle B^*_2,q_2\rangle \right) dt
\end{align*}
From this, we get the continuity of $\lVert \mathbf{q}(t)\rVert^2$ thanks to the integrability properties of $\mathbf{q}$. Weak continuity and continuity of the norm implies strong continuity, thus we have the strong continuity of $\mathbf{q}$ as a process taking values in $\cH$.  
Alternatively the strong continuity in $\cH$ of $\mathbf{q}$ follows from the results in \cite{pardoux1975equations}.

\subsection{Existence, Uniqueness and Further Results}
\label{subsec:existence}

First, it is easy to show that the solution of \eqref{Ito qg}, when it exists, it is unique. 
\begin{theorem}
    There is at most one weak solution of problem (\ref{eq:QG_ito}) in the sense of Definition~\ref{weak solution qg}.
\end{theorem}
\begin{proof}
    Let $\mathbf{q}, \ \tilde{\mathbf{q}}$ be two solutions and let $\mathbf{v}$ be their difference. Let $\bfpsi,\ \tilde{\bfpsi}$ be the corresponding streamfunctions and $\bm{\chi}$ be their difference. Let us consider the difference between two weak solutions and applying It\^{o}'s formula, which can be justified arguing as in the proof of Proposition~\ref{Ito qg} below, we get 
    \begin{align*}
        \frac{d\lVert v_1\rVert^2}{2}&=\left(\nu\langle \Delta \chi_1,\Delta v_1\rangle-\langle \nabla^{\perp}\psi_1\cdot\nabla q_1,v_1\rangle+\langle \nabla^{\perp}\tilde{\psi}_1\cdot\nabla\tilde{q}_1,v_1\rangle_{H^{-1},H^1}-\beta \langle \partial_x \chi_1,v_1\rangle \right) dt\\
        \frac{d\lVert v_2\rVert^2}{2}&=\left(\nu\langle \Delta \chi_2,\Delta v_2\rangle-\langle \nabla^{\perp}\psi_2\cdot\nabla q_2,v_2\rangle_{H^{-1},H^1}+\langle \nabla^{\perp}\tilde{\psi}_2\cdot\nabla\tilde{q}_2,v_2\rangle-\beta \langle \partial_x \chi_2,v_2\rangle-r\langle\Delta \chi_2,v_2\rangle \right)dt.
    \end{align*}
    Let us rewrite better $-\langle \nabla^{\perp}\psi_1\cdot\nabla q_1,v_1\rangle_{H^{-1},H^1}+\langle \nabla^{\perp}\tilde{\psi}_1\cdot\nabla\tilde{q}_1,v_1\rangle_{H^{-1},H^1}$, the other one is analogous. 
    \begin{align*}
        & -\langle \nabla^{\perp}\psi_1\cdot\nabla q_1,v_1\rangle_{H^{-1},H^1}+\langle \nabla^{\perp}\tilde{\psi}_1\cdot\nabla\tilde{q}_1,v_1\rangle_{H^{-1},H^1}\pm \langle \nabla^{\perp}\tilde{\psi}_1\cdot\nabla q_1,v_1\rangle_{H^{-1},H^1}\\ & = -\langle \nabla^{\perp}\chi_1\cdot q_1,v_1 \rangle_{H^{-1},H^1}-\langle \nabla^{\perp}\tilde{\psi}_1\cdot \nabla v_1,v_1 \rangle_{H^{-1},H^1}=\langle \nabla^{\perp}\chi_1\cdot \nabla v_1,q_1 \rangle_{H^{-1},H^1}.
    \end{align*}
    Multiplying the equations by $h_1$ and $h_2$ respectively and summing up, if we call $\tilde{\beta}_{1,2}=h_{1,2}\beta$, $\tilde{r}=rh_2$ we get
    \begin{align*}
         \frac{d\vertiii{\mathbf{v}}^2}{2}&=\left(\nu h_1\langle \Delta \chi_1,\Delta v_1\rangle+h_1 \langle \nabla^{\perp}\chi_1\cdot v_1,q_1 \rangle_{H^{-1},H^1}-\tilde{\beta}_1 \langle \partial_x \chi_1,v_1\rangle \right) dt\\ &+\left(\nu h_2\langle \Delta \chi_2,\Delta v_2\rangle+h_2\langle \nabla^{\perp}\chi_2\cdot v_2,q_2 \rangle_{H^{-1},H^1}-\tilde{\beta}_2 \langle \partial_x \chi_2,v_2\rangle-\tilde{r}\langle\Delta \chi_2,v_2\rangle \right)dt =: RHS dt
    \end{align*}
    Let $R(t)=\exp(-\int_0^T \eta \vertiii{\nabla \mathbf{q}(s)}^2 \ ds)$, where $\eta$ is a positive constant fixed below. Therefore we have 
    \begin{align*}
        \frac{d R(t) \vertiii{\mathbf{v}}^2}{2}&=R(t) RHS \, dt-\frac{\eta}{2} \vertiii{\mathbf{v}}^2\vertiii{\nabla \mathbf{q}}^2 R(t) \, dt 
    \end{align*}
    Exploiting relation (\ref{eq:relation_q_psi}) we obtain 
    \begin{align*}
        \frac{d R(t) \vertiii{\mathbf{v}}^2}{2}+\nu R(t) \vertiii{\nabla \mathbf{v}}^2 dt= & -\frac{\eta}{2} \vertiii{\mathbf{v}}^2\vertiii{\nabla \mathbf{q}}^2 R(t) dt  -R(t) \nu S\langle \Delta(\chi_2-\chi_1),v_1-v_2\rangle \ dt\\ & \left(-\tilde{\beta}_1R(t) \langle \partial_x \chi_1,v_1\rangle-\tilde{\beta}_2R(t) \langle \partial_2 \chi_2,v_2\rangle-\tilde{r}R(t)\lVert v_2\rVert^2\right) \,dt\\ & +\left(\tilde{r}R(t) \langle \chi_1-\chi_2,v_2\rangle+h_1R(t) \langle \nabla^{\perp}\chi_1\cdot v_1,q_1 \rangle_{H^{-1},H^1}\right.\\
        &\left.+ h_2R(t) \langle \nabla^{\perp}\chi_2\cdot v_2,q_2 \rangle_{H^{-1},H^1}\right) \,dt.
    \end{align*}
    The terms which comes from the linear part of the equations can be estimated easily, up to some constant $C$, by $R(t)\vertiii{\bfv(t)}^2$, thus we need only to analyze the nonlinear ones. Again we treat only one of the two in order to avoid repetitions.
    \begin{align*}
        \langle \nabla^{\perp}\chi_1\cdot v_1,q_1 \rangle_{H^{-1},H^1} & \leq \lVert \nabla q_1\rVert  \lVert  v_1\rVert_{L^4} \lVert  \nabla^{\perp}\chi_1\rVert_{L^4} \leq C\lVert \nabla q_1\rVert  \vertiii{\mathbf{v}}\vertiii{\nabla \mathbf{v}}\\ & \leq \frac{\nu}{4h_1}\vertiii{\nabla \mathbf{v}} ^2+C\lVert \nabla q_1\rVert^2  \vertiii{\mathbf{v}} ^2.
    \end{align*}
    Therefore, taking $\eta$ large enough we get
    \begin{align*}
        \frac{d R \vertiii{\mathbf{v}}^2 }{2}+\frac{\nu}{2}R\vertiii{\nabla \mathbf{v}}^2 dt\leq  & CR \vertiii{\mathbf{v}}^2 dt
    \end{align*}
    and the thesis follows immediately by Gronwall's Lemma.
\end{proof}

Harder task is to ensure the existence of the solutions of equations (\ref{eq:QG_ito}). To show it we need the following crucial lemma which, for an appropriate stopping time $\tau_M$, ensures convergence of the associated stopped Galerkin process $\q^N$ to the stopped process $\q$. This procedure is standard in stochastic analysis, see for example \cite{breckner2000galerkin}, \cite{razafimandimby2012strong}, \cite{luongo2022inviscid}. 
\begin{lemma}\label{crucial}
    Let $\tau_M=\operatorname{inf}\{t\in [0,T]: \vertiii{\mathbf{q}(t)}^2\geq M\}\wedge \operatorname{inf}\{t\in [0,T]: \int_0^t \vertiii{\mathbf{q}(s)}_{\cV}^2 \ ds\geq M\}\wedge T$, then 
    \begin{align*}
        \mathbbm{1}_{[0, \tau_M]}(\mathbf{q}^N-\mathbf{q})\rightarrow 0,\ \textit{ in } L^2(\Omega,L^2(0,T; \cH))
    \end{align*}
\end{lemma}
\begin{proof}
We have to show that 
\begin{equation}
\label{stopped processes difference}
    \E \int_0^T \mathbbm{1}_{[0, \tau_M]}(s)\vertiii{\q^N(s) - \q(s)}^2 \,ds
\end{equation}
converges to zero in $N$. Let us call $\Tilde{\psi}_j^N=P^N \psi_j\ \Tilde{q}_j^N=P^N q_j,\ j\in {1,2}$ and $\tilde{\bfpsi}^N, \ \tilde{\mathbf{q}}^N$ satisfy relation (\ref{eq:relation_q_psi}). Then, by the triangular inequality 
\begin{equation*}
    \eqref{stopped processes difference} \leq  2\E \int_0^T \mathbbm{1}_{[0, \tau_M]}(s)\vertiii{\tq^N(s) - \q(s)}^2 \,ds +2  \E \int_0^T \mathbbm{1}_{[0, \tau_M]}(s)\vertiii{\tq^N(s) - \q^N(s)}^2 \,ds
\end{equation*}
Thanks to the properties of the projector $P^N$ and dominated convergence theorem, it can be shown that $\Tilde{\mathbf{q}}^N\rightarrow \mathbf{q}$ in $L^2(\Omega, L^2(0,T; \cV))\cap L^2(\Omega, C(0,T; \cH)), $ and also in weaker topologies. Therefore, we are left to show the convergence of
\begin{equation}
     \E \int_0^{\tau_M} \vertiii{ \tq^N(s) - \q^N(s)}^2 \,ds .
\end{equation}
%
Let us start by taking the difference of \eqref{limit equation} and \eqref{weak formulation Galerkin vorticity}, pair the two components respectively with $h_1\left( \tilde{q}_1^N - q_1^N \right)$ and $h_2\left( \tilde{q}_2^N - q_2^N \right)$, and add them together to get an equation for $\vertiii{\tq^N - \q^N}$. With It\^{o} formula and some elementary manipulation using \eqref{eq:difference relation}, we have 
\begin{equation}
\begin{split}
      \frac{1}{2}d \vertiii{\tilde{\mathbf{q}}^N-\mathbf{q}^N}^2+\nu \vertiii{\nabla (\tilde{\mathbf{q}}^N-\mathbf{q}^N)}^2 \ dt =  
    -h_1\langle B_1^*-B_1^N,\Tilde{q}_1^N-q_1^N\rangle_{H^{-1},H^1}\,dt  -h_2\langle B_2^*-B_2^N,\Tilde{q}_2^N-q_2^N\rangle_{H^{-1},H^1}\,dt \\ 
    +  \nu S \lVert( \Tilde{q}_1^N-\Tilde{q}_2^N)-(q_1^N-q_2^N)\rVert^2 - \nu S(S_1+S_2)\lVert(-\Delta+S_1+S_2)^{-1/2}\left( (\Tilde{q}_1^N-\Tilde{q}_2^N)-(q_1^N-q_2^N)  \right) \rVert^2\\ 
      -\Tilde{\beta_1}\langle \partial_x(\psi_1-\psi_1^N),\Tilde{q}_1^N-q_1^N\rangle\ dt-\Tilde{\beta_2}\langle \partial_x(\psi_2-\psi_2^N),\Tilde{q}_2^N-q_2^N\rangle\ dt
       -S_1\sum_{k=1}^K\langle q_1-q_1^N,G^k_1(\Tilde{q}_1^N-q_1^N)\rangle dW^{1,k}_t \\
      -S_2\sum_{k=1}^K\langle q_2-q_2^N,G^k_2(\Tilde{q}_2^N-q_2^N)\rangle dW^{2,k}_t
       +S_1\langle q_1,F_1(\tilde{q}_1^N-q_1^N)\rangle\ dt-S_1\langle q_1^N,F_1^N(\tilde{q}_1^N-q_1^N)\rangle\ dt  \\
      +S_2\langle q_2,F_2(\tilde{q}_2^N-q_2^N)\rangle\ dt-S_2\langle q_2^N,F_2^N(\tilde{q}_2^N-q_2^N)\rangle\ dt  
      -\Tilde{r}\langle \Delta( \Tilde{\psi}_2^N-\psi_2^N),\Tilde{q}_2^N-q_2^N\rangle\ dt\\ 
      +\frac{S_1}{2}\sum_{k=1}^K\sum_{i=1}^N\langle q_1-q_1^N,G_1^k e_i\rangle^2 \ dt+\frac{S_2}{2}\sum_{k=1}^K\sum_{i=1}^N\langle q_2-q_2^N,G_2^k e_i\rangle^2 \ dt.
\end{split}\label{existence RHS}
\end{equation}
where we denoted $\Tilde{ \beta}_{1,2}:=h_{1,2}\beta,\ \Tilde{r}:=h_2r$ and $
    B_{1,2}^N=\nabla^{\perp}\psi_{1,2}^N\cdot \nabla q_{1,2}^N,
$.
    Now we can estimate the RHS in \eqref{existence RHS}. 
    In the following $C$ will denote a generic constant independent of $N$. 
    \begin{align}\label{eq:lemma_estimate1}
        \lVert \langle \Delta( \Tilde{\psi}_2^N-\psi_2^N),\Tilde{q}_2^N-q_2^N\rangle\rVert & \leq \lVert \Delta( \Tilde{\psi}_2^N-\psi_2^N)\rVert \lVert \Tilde{q}_2^N-q_2^N\lVert \rVert\leq C\vertiii{\tilde{\mathbf{q}}^N-\mathbf{q}^N} ^2;
    \end{align}
    \begin{align}
        \lVert( \Tilde{q}_1^N-\Tilde{q}_2^N)-(q_1^N-q_2^N)   \rVert^2\leq C\vertiii{\tilde{\mathbf{q}}^N-\mathbf{q}^N}^2;\label{eq:lemma_estimate2}
    \end{align}
    \begin{align}
        \lVert\Tilde{\beta_1}\langle \partial_x(\psi_1-\psi_1^N),\Tilde{q}_1^N-q_1^N\rangle\ +\Tilde{\beta_2}\langle \partial_x(\psi_2-\psi_2^N),\Tilde{q}_2^N-q_2^N\rangle\ \rVert& \leq C \vertiii{\mathbf{q}-\mathbf{q}^N}\vertiii{\tilde{\mathbf{q}}^N-\mathbf{q}^N} \nonumber\\ 
        &\leq C\vertiii{\tilde{\mathbf{q}}^N-\mathbf{q}^N}^2+C\vertiii{\mathbf{q}-\tilde{\mathbf{q}}^N}\vertiii{\tilde{\mathbf{q}}^N-\mathbf{q}^N} \nonumber \\ 
        & \leq C\vertiii{\tilde{\mathbf{q}}^N-\mathbf{q}^N}^2+C\vertiii{\mathbf{q}-\tilde{\mathbf{q}}^N} ^2.\label{eq:lemma_estimate3}
    \end{align}
    Next, to better understand the behavior of the terms
    \begin{align}
        2\langle q_1,F_1(\title{q}_1^N-q_1^N)\rangle - 2 \langle q_1^N,F_1^N(\title{q}_1^N-q_1^N)\rangle+ \sum_{k\in K}\sum_{i=1}^N\langle q_1-q_1^N,G^{k}_1e_i\rangle^2 \label{F1+F1N+G}
    \end{align}
    we will first write them in an equivalent form. By definition of $F_1$ \eqref{eq:F_j(q)} and Green's theorem we have 
    \begin{align*}
        2\langle q_1,F_1(\tilde{q}_1^N-q_1^N)\rangle-2\langle q_1^N,F_1^N(\tilde{q}_1^N-q_1^N)\rangle 
        = -\sum_{k=1}^K\langle \bm{\sigma}_{1,k}\cdot\nabla(\Tilde{q}_1^N-q_1^N), \bm{\sigma}_{1,k}\cdot\nabla q_1\rangle \\+\sum_{k=1}^K\langle P^N\left(\bm{\sigma}_{1,k}\cdot\nabla(\Tilde{q}_1^N-q_1^N)\right), \bm{\sigma}_{1,k}\cdot\nabla q_1^N\rangle\\
      = -\sum_{k=1}^K\langle \bm{\sigma}_{1,k}\cdot\nabla(\Tilde{q}_1^N-q_1^N), \bm{\sigma}_{1,k}\cdot\nabla q_1 - P^N\left(\bm{\sigma}_{1,k}\cdot\nabla q_1^N\right)\rangle.
        \end{align*}
        Similarly, recalling the definition \eqref{eq:Gkj(q)} of $G^k_1$, $k = 1, \ldots, K$, we have 
    \begin{align*}
        \sum_{k=1}^K\sum_{i=1}^N\langle q_1-q_1^N,G^k_1 e_i\rangle^2&=\sum_{k=1}^K\sum_{i=1}^N\langle \bm{\sigma}_{1,k}\cdot \nabla q_1-\bm{\sigma}_{1,k}\cdot \nabla q_1^N ,e_i\rangle^2\\ & =  \sum_{k=1}^K\langle \bm{\sigma}_{1,k}\cdot \nabla q_1-P^N(\bm{\sigma}_{1,k}\cdot \nabla q_1^N), P^N(\bm{\sigma}_{1,k}\cdot \nabla q_1)-P^N(\bm{\sigma}_{1,k}\cdot \nabla q_1^N)\rangle.
    \end{align*}
    Therefore, by some additional manipulations, it can be shown that 
    \begin{multline*}
        %
        %
        \eqref{F1+F1N+G} 
        = -\sum_{k=1}^K\langle \bm{\sigma}_{1,k}\cdot\nabla q_1-P^N(\bm{\sigma}_{1,k}\cdot \nabla q_1^N), (I-P^N)(\bm{\sigma}_{1,k}\cdot\nabla q_1) \rangle  
        + \sum_{k=1}^K\langle \bm{\sigma}_{1,k}\cdot\nabla q_1-P^N(\bm{\sigma}_{1,k}\cdot \nabla q_1^N), \bm{\sigma}_{1,k}\cdot\nabla (q_1-\Tilde{q}_1^N) \rangle\\
         + \sum_{k=1}^K\langle \bm{\sigma}_{1,k}\cdot\nabla q_1-P^N(\bm{\sigma}_{1,k}\cdot \nabla q_1^N), (I-P^N)(\bm{\sigma}_{1,k}\cdot \nabla q_1^N)\rangle.
    \end{multline*}
    We can now estimate all terms by Cauchy-Schwartz inequality to get
    \begin{multline}\label{eq:lemma_estimateF1G}
        \eqref{F1+F1N+G}  \leq \sum_{k=1}^K \lVert \bm{\sigma}_{1,k}\cdot\nabla q_1-P^N(\bm{\sigma}_{1,k}\cdot \nabla q_1^N)\rVert \lVert (I-P^N)(\bm{\sigma}_{1,k}\cdot\nabla q_1)  \rVert \\ +\sum_{k=1}^K C\lVert \bm{\sigma}_{1,k}\cdot\nabla q_1-P^N(\bm{\sigma}_{1,k}\cdot \nabla q_1^N)\rVert \lVert \nabla(q_1-\Tilde{q}_1^N)  \rVert
         +\sum_{k=1}^KC\lVert (I-P^N)(\bm{\sigma}_{1,k}\cdot \nabla q_1) \rVert \lVert \nabla q_1^N\rVert.
    \end{multline}
    The same estimate holds for the terms in \eqref{existence RHS} relative to the second component. 
    
We now move on to treating the nonlinear term
\begin{equation*}
     - h_1\langle B_1^* - B_1^N,\Tilde{q}_1^N-q_1^N\rangle_{H^{-1},H^1}\,dt  - h_2\langle B_2^* - B_2^N,\Tilde{q}_2^N-q_2^N\rangle_{H^{-1},H^1}\,dt 
\end{equation*}
in \eqref{existence RHS} and again we only treat $\langle B_1^* - B_1^N,\Tilde{q}_1^N-q_1^N\rangle_{H^{-1},H^1}$ as for the second component the argument is identical. 
By definition $B_1^N = \nabla^{\perp} \psi_1^N\cdot \nabla q^N_1$, hence 
\begin{multline}
    - h_1 \langle B_1^* - B_1^N,\Tilde{q}_1^N-q_1^N\rangle_{H^{-1},H^1} 
    = h_1 \langle \nabla^{\perp}(\tilde{\psi}_1^N-\psi_1^N)\cdot \nabla(\tilde{q}_1^N-q_1^N),\tilde{q}_1^N\rangle_{H^{-1},H^1}  \\ 
    +h_1\langle \nabla^{\perp}\tilde{\psi}_1^N\cdot \nabla \tilde{q}_1^N - B_1^*, \, \tilde{q}_1^N-q_1^N\rangle_{H^{-1},H^1}.\label{eq:lemma_nonlinearity}
\end{multline}
We can estimate the first term on the right hand side of the last display as follows
\begin{align*}
    \lvert  \langle \nabla^{\perp}(\tilde{\psi}_1^N-\psi_1^N)\cdot \nabla(\tilde{q}_1^N-q_1^N),\tilde{q}_1^N\rangle_{H^{-1},H^1}  \rvert & \leq \lVert \nabla \Tilde{q}_1^N\rVert \lVert \tilde{q}_1^N-q_1^N\rVert_{L^4}\lVert  \lVert \nabla^{\perp}(\tilde{\psi}_1^N-\psi_1^N) \rVert_{L^4}\\ & \leq C\lVert \nabla {q}_1\rVert   \lVert \tilde{q}_1^N-q_1^N\rVert_{L^4}^2
\end{align*}
and, by Ladyzhenskaya's and Young's inequalities we have 
\begin{align}
     \leq C\lVert \nabla {q}_1\rVert   \lVert \tilde{q}_1^N-q_1^N\rVert\lVert \nabla(\tilde{q}_1^N-q_1^N)\rVert \leq C_\nu \lVert \nabla q_1\rVert^2 \lVert \tilde{q}_1^N-q_1^N\rVert^2+\frac{\nu}{2}\lVert \nabla (\tilde{q}_1^N-q_1^N)\rVert^2.\label{eq:lemma_estimateB1}
\end{align}
The second term in \eqref{eq:lemma_nonlinearity} can be rewritten as 
\begin{align*}
\langle \nabla^{\perp}\tilde{\psi}_1^N\cdot\nabla\tilde{q}_1^N - B_1^*,\, \tilde{q}_1^N-q_1^N\rangle_{H^{-1},H^1} &= \langle \nabla^{\perp}{\psi}_1\cdot\nabla{q}_1 - B_1^* ,\, \tilde{q}_1^N-q_1^N\rangle_{H^{-1},H^1} \\ &- \langle \nabla^{\perp}{\psi}_1\cdot\nabla{q}_1-\nabla^{\perp}\tilde{\psi}_1^N\cdot\nabla\tilde{q}_1^N,\tilde{q}_1^N-q_1^N\rangle_{H^{-1},H^1} .
\end{align*}
Since $\tilde{\q}^N-\q^N\rightharpoonup 0$ in $L^2(\Omega;L^2(0,T; \cV))$, the expected value of first term will go to $0$ easily. For what concern the second one, adding and subtracting $\langle \nabla^{\perp}{\psi}_1\cdot\nabla\tilde{q}_1^N,\tilde{q}_1^N-q_1^N\rangle_{H^{-1},H^1}$ we have
\begin{align*}
  \langle \nabla^{\perp}{\psi}_1\cdot\nabla{q}_1-\nabla^{\perp}\tilde{\psi}_1^N\cdot\nabla\tilde{q}_1^N,\tilde{q}_1^N-q_1^N\rangle_{H^{-1},H^1} & = \langle \nabla^{\perp}{\psi}_1\cdot(\nabla{q}_1-\nabla\tilde{q}_1^N),\tilde{q}_1^N-q_1^N\rangle_{H^{-1},H^1}\\ & + \langle (\nabla^{\perp}{\psi}_1-\nabla^{\perp}{\tilde{\psi}_1}^N)\cdot\nabla\tilde{q}_1^N,\tilde{q}_1^N-q_1^N\rangle_{H^{-1},H^1}
\end{align*}
so that, again by Ladyzhenskaya's inequality,
\begin{align*}
    \leq \| \nabla^{\perp}\psi_1\|_{L^4} \| q_1-\tilde{q}_1^N \|^{1/2} \| \nabla( q_1-\tilde{q}_1^N)\rVert^{1/2} \| \nabla (\tilde{q}_1^N-q_1^N)\|  + \lVert \nabla^{\perp}\psi_1-\nabla^{\perp}\tilde{\psi}_1^N\rVert_{L^4}\lVert \tilde{q}_1^N\rVert^{1/2}\lVert \nabla \tilde{q}_1^N\rVert^{1/2}\lVert \nabla (\tilde{q}_1^N-q_1^N)\rVert.
\end{align*}
In summary, we showed
\begin{align}\label{eq:lemma_estimateB1B*}
\begin{split}
    \langle \nabla^{\perp}\tilde{\psi}_1^N\cdot\nabla\tilde{q}_1^N - B_1^*,\, \tilde{q}_1^N-q_1^N\rangle =
     &  \langle \nabla^{\perp}{\psi}_1\cdot\nabla{q}_1 - B_1^* ,\, \tilde{q}_1^N-q_1^N\rangle \\ 
     &+ \| \nabla^{\perp}\psi_1\|_{L^4} \| q_1-\tilde{q}_1^N \|^{1/2} \| \nabla( q_1-\tilde{q}_1^N)\rVert^{1/2} \| \nabla (\tilde{q}_1^N-q_1^N)\|  \\
     &+ \lVert \nabla^{\perp}\psi_1-\nabla^{\perp}\tilde{\psi}_1^N\rVert_{L^4}\lVert \tilde{q}_1^N\rVert^{1/2}\lVert \nabla \tilde{q}_1^N\rVert^{1/2}\lVert \nabla (\tilde{q}_1^N-q_1^N)\rVert. 
     \end{split}
\end{align}
%

Consider now the auxiliary function 
\begin{equation}
    R(t):=  \tfrac{1}{2}\exp(-\eta_1 t-\eta_2\int_0^t \vertiii{\nabla \mathbf{q}(s)}^2\ ds),
\end{equation}
with $\eta_1$ and $\eta_2$ two positive constants to be defined later, and let us then compute via the It\^{o}'s formula $R(t)\vertiii{\tq^N(t) - \q^N(t)}^2$.
 We exploit previous estimates \eqref{eq:lemma_estimate1},\eqref{eq:lemma_estimate2},\eqref{eq:lemma_estimate3},\eqref{eq:lemma_estimateF1G} and \eqref{eq:lemma_estimateB1} and take the expected value for $t=\tau_M$ obtaining
\begin{align}\begin{split}
    &\E \left[\tfrac{1}{2}R(\tau_M)\vertiii{\tilde{\mathbf{q}}^N(\tau_M)-\mathbf{q}^N(\tau_M)}^2\right]+\nu\E \left[ \int_0^{\tau_M} R(s)\vertiii{\nabla(\tilde{\mathbf{q}}^N(s)-\mathbf{q}^N(s))} ^2\, ds \right]\leq\\
    &  -\frac{\eta_1}{2}\E \left[\int_0^{\tau_M}R(s)\vertiii{\tilde{\mathbf{q}}^N(s)-\mathbf{q}^N(s)} ^2 \ ds \right] -\frac{\eta_2}{2}\E \left[\int_0^{\tau_M}R(s)\vertiii{\nabla \mathbf{q}(s)}^2 \vertiii{\tilde{\mathbf{q}}^N(s)-\mathbf{q}^N(s)}^2 \ ds \right]\\ 
    & +h_1\left\lvert\E \left[\int_0^{\tau_M}R(s) \langle B_1^*(s)-\nabla^{\perp}\tilde{\psi}_1^N(s)\cdot\nabla\tilde{q}_1^N(s),\tilde{q}_1^N(s)-q_1^N(s)\rangle\   \ ds\right]\right\rvert\\ 
    & + h_2\left\lvert\E \left[\int_0^{\tau_M}R(s) \langle B_2^*(s)-\nabla^{\perp}\tilde{\psi}_2^N(s)\cdot\nabla\tilde{q}_2^N(s),\tilde{q}_2^N(s)-q_2^N(s)\rangle\   \ ds\right]\right\rvert\\ 
    & + C\E \left[\int_0^{\tau_M} R(s)\vertiii{\nabla \q(s)}^2 \vertiii{\tilde{\mathbf{q}}^N(s)-\mathbf{q}^N(s)}^2\ ds\right] +\frac{\nu}{2}\E \left[\int_0^{\tau_M} R(s) \vertiii{\nabla\tilde{\mathbf{q}}(s)^N-\nabla \mathbf{q}^N(s)}^2\ ds\right]\\ 
    & +C\E \left[\int_0^{\tau_M}R(s) \vertiii{\tilde{\mathbf{q}}^N(s)-\mathbf{q}^N(s)}^2\ ds \right]
     +C\E \left[\int_0^{\tau_M}R(s) \vertiii{\tilde{\mathbf{q}}^N(s)-\mathbf{q}(s)}^2\ ds \right]\\
    & +\sum_{k=1}^K \E \left[\int_0^{\tau_M}\lVert \bm{\sigma}_{1,k}\cdot\nabla q_1(s)-P^N(\bm{\sigma}_{1,k}\cdot \nabla q_1^N(s))\rVert \lVert (I-P^N)(\bm{\sigma}_{1,k}\cdot\nabla q_1(s))  \rVert\ ds\right]\\ 
    &+C\sum_{k=1}^K \E \left[\int_0^{\tau_M}\lVert \bm{\sigma}_{1,k}\cdot\nabla q_1(s)-P^N(\bm{\sigma}_{1,k}\cdot \nabla q_1^N(s))\rVert \lVert \nabla(q_1(s)-\Tilde{q}_1^N(s))  \rVert\ ds\right]\\
    & +C\sum_{k=1}^K\E \left[\int_0^{\tau_M}\lVert (I-P^N)(\bm{\sigma}_{1,k}\cdot \nabla q_1(s)) \rVert \lVert \nabla q_1^N(s)\rVert \ ds\right]\\
    &+\sum_{k=1}^K \E \left[\int_0^{\tau_M}\lVert \bm{\sigma}_{2,k}\cdot\nabla q_2(s)-P^N(\bm{\sigma}_{2,k}\cdot \nabla q_2^N(s))\rVert \lVert (I-P^N)(\bm{\sigma}_{2,k}\cdot\nabla q_2(s))  \rVert\ ds\right]\\
    &+C\sum_{k=1}^K \E \left[\int_0^{\tau_M}\lVert \bm{\sigma}_{2,k}\cdot\nabla q_2(s)-P^N(\bm{\sigma}_{2,k}\cdot \nabla q_2^N(s))\rVert \lVert \nabla(q_2(s)-\Tilde{q}_2^N(s))  \rVert\ ds\right]\\
    & +C\sum_{k=1}^K\E \left[\int_0^{\tau_M}\lVert (I-P^N)(\bm{\sigma}_{2,k}\cdot \nabla q_2(s)) \rVert \lVert \nabla q_2^N(s)\rVert \ ds\right].
    \end{split}\label{eq:lemmaExp_eq}
\end{align}
Taking $\eta_1,\ \eta_2$ large enough and exploiting the convergence of $\tilde{\mathbf{q}}^N$ to $\mathbf{q}$ we can neglect several terms in the right hand side. Let us consider the remaining terms. 
%
Recalling that from the weak convergence of $\mathbf{q}^N$ it follows that $\E \left[\int_0^T\vertiii{\nabla \mathbf{q}^N(s)}^2 \ ds\right]\leq C$, applying Cauchy–Schwarz inequality where it is needed, we get
 \begin{align*}
    & \sum_{k=1}^K \E \left[\int_0^{\tau_M}\lVert \bm{\sigma}_{1,k}\cdot\nabla q_1(s)-P^N(\bm{\sigma}_{1,k}\cdot \nabla q_1^N(s))\rVert \lVert (I-P^N)(\bm{\sigma}_{1,k}\cdot\nabla q_1(s))  \rVert\ ds\right]\\ 
    &+C\sum_{k=1}^K \E \left[\int_0^{\tau_M}\lVert \bm{\sigma}_{1,k}\cdot\nabla q_1(s)-P^N(\bm{\sigma}_{1,k}\cdot \nabla q_1^N(s))\rVert \lVert \nabla(q_1(s)-\Tilde{q}_1^N(s))  \rVert\ ds\right]\\ 
    & +C\sum_{k=1}^K\E \left[\int_0^{\tau_M}\lVert (I-P^N)(\bm{\sigma}_{1,k}\cdot \nabla q_1(s)) \rVert \lVert \nabla q_1^N(s)\rVert \ ds\right]\\ 
    &\leq C\sum_{k=1}^K\E \left[\int_0^{T} \lVert (I-P^N)(\sigma_{1,k}\cdot\nabla q_1(s))  \rVert^2\ ds\right]^{1/2}+ C\E \left[\int_0^{T} \lVert \nabla (q_1(s)-\tilde{q}_1^N(s))  \rVert^2\ ds\right]^{1/2}\rightarrow 0
\end{align*}
and similarly for the second component. 
Lastly, let us consider 
\[\left\lvert\E \left[\int_0^{\tau_M}R(s) \langle B_1^*(s)-\nabla^{\perp}\tilde{\psi}_1^N(s)\cdot\nabla\tilde{q}_1^N(s),\tilde{q}_1^N(s)-q_1^N(s)\rangle\   \ ds\right]\right\rvert\]
%
By \eqref{eq:lemma_estimateB1B*} we have that 
\begin{align}
\begin{split}
  & \E \left[\int_0^{\tau_M}R(s)  \langle \nabla^{\perp}\tilde{\psi}_1^N(s)\cdot\nabla\tilde{q}_1^N(s) - B_1^*(s),\, \tilde{q}_1^N(s)-q_1^N(s)\rangle\, ds \right]\\ &  \leq
       \E \left[\int_0^{\tau_M}R(s) \langle \nabla^{\perp}{\psi}_1(s)\cdot\nabla{q}_1(s) - B_1^*(s) ,\, \tilde{q}_1^N(s)-q_1^N(s)\rangle \, ds \right]\\ &
     + \E \left[\int_0^{\tau_M}R(s) \| \nabla^{\perp}\psi_1(s)\|_{L^4} \| q_1(s)-\tilde{q}_1^N(s) \|^{1/2} \| \nabla( q_1(s)-\tilde{q}_1^N(s))\rVert^{1/2} \| \nabla (\tilde{q}_1^N(s)-q_1^N(s))\|  \, ds \right]\\ &
     + \E \left[\int_0^{\tau_M}R(s) \lVert \nabla^{\perp}\psi_1(s)-\nabla^{\perp}\tilde{\psi}_1^N(s)\rVert_{L^4}\lVert \tilde{q}_1^N(s)\rVert^{1/2}\lVert \nabla \tilde{q}_1^N(s)\rVert^{1/2}\lVert \nabla (\tilde{q}_1^N(s)-q_1^N(s))\rVert\, ds \right]. 
     \end{split}
\end{align}
As already observed the first term converges to zero as $\tilde{q}^N-q^N\rightharpoonup 0$ in $L^2(\Omega;L^2(0,T; \cV))$. For what concern the two remaining terms, thanks to H\"{o}lder inequality, we have
\begin{align*}
    %
    & \leq C \E \left[\int_0^{T}\lVert q_1(s)-\tilde{q}_1^N(s)\rVert^{2}\ ds\right]^{1/4}\E \left[\int_0^{T}\lVert \nabla (q_1(s)-\tilde{q}_1^N(s))\rVert^{2}\ ds\right]^{1/4}\E \left[\int_0^{T}\lVert \nabla(q_1^N(s)-\tilde{q}_1^N(s))\rVert^{2}\ ds\right]^{1/2}\\ 
    & + C \E \left[\int_0^{T}\lVert q_1(s)-\tilde{q}_1^N(s)\rVert^{2} \ ds\right]^{1/4}\E \left[\int_0^{\tau_M}\lVert \nabla\tilde{q}_1^N(s)\rVert^{2} \ ds\right]^{1/4}\E \left[\int_0^{T}\lVert \nabla(q_1^N(s)-\tilde{q}_1^N(s))\rVert^{2} \ ds\right]^{1/2}.
\end{align*}
In the last inequality we exploit the fact that $\vertiii{q(t)}^2\leq M$, for all $t\leq \tau_M$, and $\tilde{q}^N-q^N\rightharpoonup 0$ in $L^2(\Omega;L^2(0,T; \cV))$.

In conclusion, in \eqref{eq:lemmaExp_eq} all the terms on the right hand side converge to zero as $N\to \infty$, namely we have the following relation 
\begin{align}\label{end of lemma}
    \E \left[\frac{1}{2}R(\tau_M)\vertiii{\tilde{\mathbf{q}}^N(\tau_M)-\mathbf{q}^N(\tau_M)}^2\right]+\frac{\nu}{2}\E \left[ \int_0^{\tau_M} R(s)\vertiii{\nabla(\tilde{\mathbf{q}}^N(s)-\mathbf{q}^N(s))}^2\, ds \right]\rightarrow 0.
\end{align}
From relation (\ref{end of lemma}), $R(t)\geq C_M >0$, for all $t\leq \tau_M $, and the properties of $P^N$, via triangular inequality the thesis follows.
\end{proof}

The lemma just shown allows to treat the nonlinearity $\nabla^{\perp} \psi_j^N \cdot \nabla q_j^N $, more precisely, to show that for both $j = 1, 2$
\begin{equation*}
    \nabla^{\perp} \psi_j^N \cdot \nabla q_j^N \rightharpoonup \nabla^{\perp} \psi_j \cdot \nabla q_j .
\end{equation*}

\begin{lemma}\label{lemma:nonlinearity}
Let $B_1^*, B_2^*$ be the limit processes as in \eqref{eq:defB12*}. Then 
    $B_1^{*}=\nabla^{\perp}\psi_1\cdot \nabla q_1$ and $B_2^{*}=\nabla^{\perp}\psi_2\cdot \nabla q_2$ in  $L^2(\Omega, L^2(0,T;H^{-1}(\mathcal{D})))$.
\end{lemma}
\begin{proof}
    Thanks to estimate \ref{energy L4} and \ref{energy strana} we know that $\nabla^{\perp}\psi_{1,2}\cdot \nabla q^N_{1,2}$ and $\nabla^{\perp}\psi_{1,2}^N\cdot \nabla q_{1,2}$ converge to $\nabla^{\perp}\psi_{1,2}\cdot \nabla q_{1,2}$ weakly in $L^2(\Omega;L^2(0,T;H^{-1}(\mathcal{D})))$.
    We do the explicit computations just for one of them, the others are analogous. \begin{align*}
        \E \left[\int_0^T \lVert \nabla^{\perp}\psi_{1,2}(s)\cdot \nabla q^N_{1,2}(s) \rVert_{H^{-1}}^2\ ds\right]&  \leq C 
         \E \left[\int_0^T \lVert \Delta\psi(s)\rVert^2 \lVert \nabla q^N(s) \rVert^2\ ds\right]\\ & \leq C 
         \E \left[\operatorname{sup}_{t\in [0,T]} \lVert \mathbf{q}(t)\rVert^2\int_0^T  \lVert \nabla \mathbf{q}^N(s) \rVert^2\ ds\right] \\ & \leq C\E \left[\operatorname{sup}_{t\in [0,T]} \lVert \mathbf{q}(t)\rVert^4\right]+ C\E \left[\left(\int_0^T  \lVert \nabla \mathbf{q}^N(s) \rVert^2\ ds\right)^2\right]\leq C.
    \end{align*}
    Let now $\phi \in L^{\infty}(\Omega;L^{\infty}(0,T;H^1(\mathcal{D})))$, then $\nabla \phi\cdot \nabla^{\perp}\psi_{1,2}\in L^2(\Omega;L^{2}(0,T;L^2(\mathcal{D})))$. Thus, from the convergence properties of $\mathbf{q}^N$, we have
    \begin{align*}
        \E \left[\int_0^T \langle \nabla^{\perp}\psi_{1,2}(s)\cdot\nabla q_{1,2}^N(s),\phi\rangle_{H^{-1},H^1} ds\right]&=-\E \left[\int_0^T \langle \nabla^{\perp}\psi_{1,2}(s)\cdot\nabla\phi, q_{1,2}^N(s)\rangle ds\right]\\ & \rightarrow
        -\E \left[\int_0^T \langle \nabla^{\perp}\psi_{1,2}(s)\cdot\nabla\phi, q_{1,2}(s)\rangle ds\right]\\& =\E \left[\int_0^T \nabla^{\perp}\psi_{1,2}(s)\cdot\nabla q_{1,2}(s),\phi\rangle_{H^{-1},H^1} ds\right].
    \end{align*}
    From the density of $L^{\infty}(\Omega;L^{\infty}(0,T;H^1(\mathcal{D})))$ in $L^{2}(\Omega;L^{2}(0,T;H^1(\mathcal{D})))$ and the uniform boundedness of $\nabla^{\perp}\psi_{1,2}\cdot \nabla q^N_{1,2}$ in $L^{2}(\Omega;L^{2}(0,T;H^{-1}(\mathcal{D})))$ we have the required claim. For what concern the convergence of the nonlinear term, first note that, arguing as above, the sequence $\{\nabla^{\perp}\psi^N_{1,2}\cdot \nabla q^N_{1,2}\}$ is uniformly bounded in $L^{2}(\Omega;L^{2}(0,T;H^{-1}(\mathcal{D})))$. Moreover we have  
    \begin{align*}
        \nabla^{\perp}\psi_{1,2}\cdot\nabla q_{1,2}- \nabla^{\perp}\psi_{1,2}^N\cdot\nabla q_{1,2}^N&= \nabla^{\perp}\psi_{1,2}\cdot\nabla (q_{1,2}-q_{1,2}^N)+\nabla^{\perp}\psi_{1,2}\cdot\nabla q_{1,2}^N\\ & +\nabla^{\perp}\psi_{1,2}^N\cdot\nabla (q_{1,2}-q_{1,2}^N)-\nabla^{\perp}\psi_{1,2}\cdot\nabla q_{1,2}^N=:I_1+I_2+I_3+I_4.
    \end{align*}
    Thanks to the previous observations $I_1+I_2+I_4$ converges weakly to  $0$ in $L^{2}(\Omega;L^{2}(0,T;H^{-1}(\mathcal{D})))$. For what concern $I_3$, let us take $\phi \in  L^{\infty}(\Omega;L^{\infty}(0,T;H^2(\mathcal{D})))$ and $\tau_M$ defined as in Lemma~\ref{crucial}, then we have
    \begin{align*}
        \E \left[\int_0^{\tau_M} \langle \nabla^{\perp}\psi_{1,2}^N(s)\cdot\nabla (q_{1,2}(s)-q_{1,2}^N(s)) ,\phi\rangle\ ds\right] & \leq C \E \left[\int_0^{\tau_M} \lVert \q^N(s)\rVert\lVert \q(s)-\q^N(s)\rVert\ ds\right]\rightarrow 0
    \end{align*}
    thanks to H\"{o}lder's inequality and Lemma~\ref{crucial}. Since it holds $\tau_M\nearrow T\ a.s.$, then the thesis follows, thanks to Lemma~\ref{stime energia}. Then the thesis follows by the density of $L^{\infty}(\Omega;L^{\infty}(0,T;H^2(\mathcal{D})))$ in $L^{2}(\Omega;L^{2}(0,T;H^1(\mathcal{D})))$ and the uniform boundedness of $\nabla^{\perp}\psi^N_{1,2}\cdot \nabla q^N_{1,2}$ in $^{2}(\Omega;L^{2}(0,T;H^{-1}(\mathcal{D})))$.
\end{proof}
This series of lemmas identify the nonlinear term and conclude the proof of Theorem~\ref{thm well psd qg vorticity}. Actually, thanks to some abstract results on stochastic processes it can be shown something more, namely that the full sequence $\mathbf{q}^N$ converges to $\mathbf{q}$ in $L^2(\Omega;L^2(0,T; \cV))$ and, for each $t\in [0,T]$, $\mathbf{q}^N(t)$ converges to $\mathbf{q}(t)$ in $L^2(\Omega; \cH)$. This is an easy corollary of previous result and Lemma~\ref{lemma:convergence}.
\begin{theorem}
    The entire Galerkin's sequence $\mathbf{q}^N$ satisfies 
    \begin{align*}
        \lim_{N\rightarrow +\infty}\E \left[\lVert \mathbf{q}^N(t)-\mathbf{q}(t)\rVert^2\right]=0;\\ 
        \lim_{N\rightarrow +\infty}\int_0^T\E \left[\lVert \nabla(\mathbf{q}^N(t)-\mathbf{q}(t))\rVert^2\right]\ dt=0.
    \end{align*}
\end{theorem}
\begin{proof}
    From the uniqueness of the solution of problem (\ref{eq:QG_ito}), we have that each sub-sequence $\mathbf{q}^{n_k}$ has a converging sub-sub-sequence $\mathbf{q}^{n_{k,k}}$ which satisfies \autoref{stime energia}, \autoref{crucial} and \autoref{lemma:nonlinearity}. In particular, defining $\tau_M$ as in the proof of Lemma~\ref{crucial}, from relation (\ref{end of lemma}), $R(t)\geq C_M>0$ for all $t\leq \tau_M $ and the properties of $P^N$ we have the thesis via Lemma~\ref{lemma:convergence} using $\sigma_M=\tau_M$ and $Q_N=\int_0^t\lVert \mathbf{q}(s)-\mathbf{q}^{n_{N,N}}(s)\rVert_V^2\ ds$ or $Q_N=\lVert \left(\q-\q^{n_{N,N}}\right)(t)\rVert^2$.
\end{proof}

Lastly, we show that the energy estimate \eqref{ITO} in Lemma~\ref{stime energia} continues to holds for $\mathbf{q}$, solution of problem (\ref{eq:QG_ito}). For what concern the a priori and integral estimates in \autoref{stime energia}, they straightforwardly continue to hold by the weak convergence of $\mathbf{q}^{N}$ to $\mathbf{q}$, but they can proved independently starting from the energy estimate and repeating the same steps of Lemma~\ref{stime energia}.
\begin{proposition}
\label{Ito qg}
    Given $\q$ solution of \eqref{eq:QG_ito}, then the following energy estimate holds 
    \begin{align*}
        \frac{d \vertiii{\mathbf{\mathbf{q}}} ^2}{2}+\nu \vertiii{\nabla \mathbf{\mathbf{q}}}^2 dt& =\left(-\beta h_1\langle \partial_x \psi_1,q_1\rangle-\beta h_2\langle \partial_x \psi_2,q_2\rangle+h_1\langle F,q_1\rangle-rh_2 \lVert q_2\rVert^2+Sr\langle \psi_1-\psi_2,q_2\rangle\right) dt\notag\\ & +\left(S\nu\lVert q_1-q_2\rVert^2+S\nu(S_1+S_2)\langle \psi_1-\psi_2,q_1-q_2\rangle\right)dt.
    \end{align*}
\end{proposition}
\begin{proof}
    Let $\Tilde{\mathbf{q}}^N$ be defined as in Lemma~\ref{crucial}. We already know by the properties of the projector $P^N$ that $\Tilde{\mathbf{q}}^N\rightarrow \mathbf{q}$ in $L^2(0,T; \cV)\cap C(0,T; \cH)\ \mathbb{P}$-a.s.\
    As $\mathbf{q}$ satisfies the weak formulation of \eqref{eq:QG_ito} with test functions $e_i$ we have
    \begin{align*}
        \left\langle q_1(t)  ,e_i\right\rangle  &  =\left\langle
        q_{1,0},e_i\right\rangle +\int_{0}^{t}\left\langle \Delta \psi_1(s)
        , \nu\Delta  e_i\right\rangle ds+\int_0^t \langle q_1(s), F_1 e_i\rangle\ ds+\int_0^t \langle q_1(s), \nabla^{\perp}\psi_1(s)\cdot\nabla e_i\rangle ds \\
        &  +\int_{0}^{t}\left\langle \beta \partial_x\psi_1\left (s\right)  ,e_i\right\rangle
        ds +\int_{0}^{t}\left\langle F(s)  ,e_i\right\rangle
        ds-\sum_{k=1}^K\int_{0}^{t}\left\langle q_1(s)  ,G^k_1e_i\right\rangle dW_{s}^{1,k}\\ 
        \left\langle q_2(t)  ,e_i\right\rangle  &  =\left\langle
        q_{2,0},e_i\right\rangle +\int_{0}^{t}\left\langle \Delta \psi_2(s)
        , \nu\Delta  e_i\right\rangle ds+\int_0^t \langle q_2(s), F_2 e_i\rangle\ ds+\int_0^t \langle q_2(s), \nabla^{\perp}\psi_2(s)\cdot\nabla e_i\rangle ds \\
        &  +\int_{0}^{t}\left\langle \beta \partial_x\psi_2\left (s\right)  ,e_i\right\rangle
        ds -r\int_{0}^{t}\left\langle \Delta \psi_2(s)  ,e_i\right\rangle
        ds-\sum_{k=1}^K\int_{0}^{t}\left\langle q_2(s)  ,G^k_2e_i\right\rangle dW_{s}^{2,k}. %
    \end{align*}
    Multiplying each equation by $e_i$ and summing up, we get
    \begin{align*}
    \begin{split}
        d\tilde{q_1}^N&=\nu\Delta^2 \tilde{\psi}_1^N dt+ \sum_{i=1}^N\left\langle\nabla^{\perp}\psi_1\cdot \nabla e_i, q_1\right\rangle e_i dt+\sum_{i=1}^N \langle F-\beta\partial_x \psi_1,e_i\rangle e_i dt\\ &  - \sum_{k=1}^K\sum_{i=1}^N\langle G^{k}_1e_i, q_1\rangle\  dW^{1,k}  +\sum_{i=1}^N \langle q_1,F_1 e_i\rangle e_i\, dt
        \end{split}\\
    \begin{split}
        d\tilde{q_2}^N&=\nu\Delta^2 \tilde{\psi}_2^N dt-r\Delta \tilde{\psi}_2^N dt+ \sum_{i=1}^N\left\langle\nabla^{\perp}\psi_2\cdot \nabla e_i, q_2\right\rangle e_i dt-\beta\sum_{i=1}^N \langle\partial_x \psi_2,e_i\rangle e_i dt\\ &  - \sum_{k=1}^K\sum_{i=1}^N\langle G^{k}_2e_i, q_2\rangle\  dW^{2,k}  +\sum_{i=1}^N \langle q_2,F_2 e_i\rangle e_i\, dt .
        \end{split}
    \end{align*}
    Now we can apply the It\^{o}'s formula to the process $\frac{\vertiii{\Tilde{\mathbf{q}}^N}}{2}^2=\frac{h_1 \lVert \tilde{q}_1^N\rVert^2+h_2 \lVert \tilde{q}_2^N\rVert^2}{2}$
    obtaining, thanks to the relations (\ref{eq:relation_q_psi}) and \eqref{eq:difference relation},
    \begin{align*}
            \frac{d \vertiii{{\mathbf{\tilde{q}}^N}}^2}{2}+\nu\vertiii{\nabla {\mathbf{\tilde{q}}^N}}^2 dt& =\left(-\beta h_1\langle \partial_x \psi_1^N,\tilde{q}_1^N\rangle-\beta h_2\langle \partial_x \psi_2^N,\tilde{q}_2^N\rangle+h_1\langle F,\tilde{q}_1^N\rangle-rh_2 \lVert \tilde{q}_2^N\rVert^2+Sr\langle \tilde{\psi}_1^N-\tilde{\psi}_2^N,\tilde{q}_2^N\rangle\right) dt\\ & +\left(S\nu\lVert \tilde{q}_1^N-\tilde{q}_2^N\rVert^2+\nu S(S_1+S_2)\langle \tilde{\psi}_1^N-\tilde{\psi}_2^N,\tilde{q}_1^N-\tilde{q}_2^N\rangle\right)dt\\ & +h_1\langle q_1, F_1\tilde{q}_1^N\rangle dt+h_2\langle q_2, F_2\tilde{q}_2^N\rangle dt+\frac{h_1}{2}\sum_{k=1}^K\sum_{i=1}^N \langle G_1^k e_i,q_1\rangle^2 dt+\frac{h_2}{2}\sum_{k=1}^K\sum_{i=1}^N \langle G_2^ke_i,q_2\rangle^2 dt\\ & -h_1\sum_{k=1}^K\langle q_1, G_1^K\tilde{q}_1^N\rangle dW^{1,k}-h_2\sum_{k=1}^K\langle q_2, G_1^K\tilde{q}_2^N\rangle dW^{2,k}.
    \end{align*}
    Then, thanks to the properties of the projector $P^N$ we get easily the desired formula. The only thing we need to prove is that 
    \begin{align*}
        h_1\langle q_1, F_1\tilde{q}_1^N\rangle+h_2\langle q_2, F_2\tilde{q}_2^N\rangle +\frac{h_1}{2}\sum_{k=1}^K\sum_{i=1}^N \langle G_1^k e_i,q_1\rangle^2 +\frac{h_2}{2}\sum_{k=1}^K\sum_{i=1}^N \langle G_2^ke_i,q_2\rangle^2 \rightarrow 0.
    \end{align*}
    The last relation is true, in fact for each $k\in \{1,\dots, K\},\ j\in \{1,2\}$ we have
    \begin{align*}
        &\sum_{i=1}^N \langle q_j(s),\bm{\sigma}_{j,k}\cdot \nabla e_i\rangle^2+\langle q_j(s),\bm{\sigma}_{j,k}\cdot\nabla\left(\bm{\sigma}_{j,k}\cdot\nabla \Tilde{q}_j^N\right)\rangle\\
        &= -\langle q_j(s),\bm{\sigma}_{j,k}\cdot \nabla \left(P^N(\bm{\sigma}_{j,k}\cdot \nabla q_j(s))\right)\rangle+\langle q_j(s),\bm{\sigma}_{j,k}\cdot\nabla\left(\bm{\sigma}_{j,k}\cdot\nabla \Tilde{q}_j^N\right)\rangle\\ 
        &=\langle \bm{\sigma}_{j,k}\cdot \nabla q_j(s),P^N\left(\bm{\sigma}_{j,k}\cdot \nabla q_j(s)\right)\rangle-\langle \bm{\sigma}_{j,k}\cdot\nabla q_j(s),\bm{\sigma}_{j,k}\cdot\nabla \Tilde{q}_j^N\rangle\rightarrow 0.
    \end{align*}
\end{proof}

\section{Enhanced dissipation by transport}\label{sec:viscosity}
%
In this section we will focus first on the proof of \autoref{thm galeati} and then, in \autoref{sec long time}, we will study the long time behaviour of the solutions of \eqref{eq:QG_ito GL} showing that \autoref{thm:convergence deterministic} holds.
%
\subsection{Convergence to the Deterministic Evolution Model}
\label{subsec:convergence}
%
The first step in order to prove Theorem~\ref{thm galeati}, is to rewrite equations (\ref{eq:QG_ito GL})-(\ref{eq:QG_deterministic GL}) in an alternative way, in order to introduce explicitly the term $\bm{\Delta}\q$.
For this reason, using once more the relation (\ref{eq:relation_q_psi}), we express $\bm{\Delta}^2 \bfpsi$ as \begin{align*}
     \bm{\Delta}^2 \bfpsi &=  \mathbf{\Delta} \left(  \q - M\bfpsi\right) = \bm{\Delta} \q+\mathbf{\Delta}M(-\mathbf{\Delta}-M)^{-1}\q.
\end{align*}
Thus if we call $\textbf{F}=\begin{bmatrix}F\\ -r\Delta \psi_2\end{bmatrix}$, equation (\ref{eq:QG_ito GL}) can be rewritten as
\begin{align}
\label{compact galeati stochastic}
\begin{split}
    d\q=\left((\kappa+\nu)\mathbf{\Delta} \q+\nu\mathbf{\Delta}M(-\mathbf{\Delta}-M)^{-1}\q-\nabla^{\perp}\bfpsi\cdot\nabla\q -\beta \partial_x\bfpsi+\textbf{F}\right)dt+\sqrt{2\kappa}\sum_{k\in K}\mathbf{a}_{k}e_k\cdot \nabla \q dW^{k}\\
    \bfpsi=-(-\mathbf{\Delta}-M)^{-1}\q\\ 
    \q(0)=\q_{0}, 
\end{split}    
\end{align}
where we denoted by 
\begin{align*}
    \mathbf{a}_{k}e_k\cdot \nabla \q dW^{k}:=\begin{bmatrix}\mathbf{a}_{1,k}e_k\cdot \nabla q_1 dW^{1,k} \\ \mathbf{a}_{2,k}e_k\cdot \nabla q_2 dW^{2,k}\end{bmatrix}
\end{align*}
Similarly, if we call $\Bar{\textbf{F}}=\begin{bmatrix}F\\ -r\Delta \bar\psi_2\end{bmatrix}$, the deterministic equation (\ref{eq:QG_deterministic GL}) can be rewritten as
\begin{align}\label{compact galeati deterministic}
\begin{split}
    d\Bar{\q}=\left((\kappa+\nu)\mathbf{\Delta} \Bar{\q}+\nu\mathbf{\Delta}M(-\mathbf{\Delta}-M)^{-1}\Bar{\q}-\nabla^{\perp}\bar\bfpsi\cdot\nabla\bar\q -\beta \partial_x\bar\bfpsi+\Bar{\textbf{F}}\right)dt\\
    \Bar{\bfpsi}=-(-\mathbf{\Delta}-M)^{-1}\Bar{\q}\\ 
    \Bar{\q}(0)=\q_{0}. 
\end{split}    
\end{align}
First, we want to show that the weak solution $\q$ of \eqref{eq:QG_ito GL} satisfies a mild formulation. 
Denote the stochastic integral and stochastic convolution as
\begin{align}
    \mathbf{M}(t):=\sqrt{2\kappa}\sum_{k\in K}\int_0^t \mathbf{a}_{k}e_k\cdot \nabla\q(s) dW^k_s  \label{eq:defM(t)}\\
    \mathbf{Z}(t):=\sqrt{2\kappa}\sum_{k\in K}\int_0^t  e^{(\kappa+\nu) (t-s)\mathbf{\Delta}}\left(\mathbf{a}_{k}e_k\cdot \nabla\q(s)\right)dW^k_s \label{eq:defZ(t)}.
\end{align}
Thanks to the results of Section \ref{sec well posed} for the stochastic system the following relations can be shown to hold
\begin{align}
    &\operatorname{sup}_{t\in[0,T]}\lVert\mathbf{q}(t)\rVert^2 \leq C\left(\lVert\mathbf{q}_0\rVert^2 +\int_0^T \lVert F(s)\rVert^2 ds \right) e^{CT}=:R_T^2,\ \as \label{eq:defR_T}\\
    &\operatorname{sup}_{t\in[0,T]}\lVert\bar{\q}(t)\rVert^2 \leq R_T^2 \label{eq:sup q bar}\\
    & \int_0^T\lVert \nabla {\q}(s)\rVert^2 \ ds \leq \frac{C}{\nu} \left(TR_T^2+\int_0^T\lVert F(s)\rVert^2\ ds\right) \label{eq:int nabla q}\\
    &\int_0^T\lVert \nabla \bar{\q}(s)\rVert^2 \ ds \leq \frac{C}{\kappa+\nu} \left(TR_T^2+\int_0^T\lVert F(s)\rVert^2\ ds\right).\label{eq:int nabla q bar}
\end{align}
The constants $C,R_T$ here above depend from all the parameters of the model (i.e. $T,\nu,r,F,\q_0,M,\beta$), except for the parameters of the noise. Thanks to these estimates, Assumption~2.4 of \cite{flandoli2021quantitative} holds. Thus, thanks to Corollary~2.6 in \cite{flandoli2021quantitative}, the stochastic integral \eqref{eq:defM(t)} and the stochastic convolution \eqref{eq:defZ(t)}
are well defined and have the regularity prescribed by the Lemma below. 
\begin{lemma}
\label{regularity stochastic convolution}
    Given the processes \eqref{eq:defM(t)} and \eqref{eq:defZ(t)} the following statements hold true:
    \begin{itemize}
        \item[(i)] $\mathbf{M}(t)$ is a continuous martingale with values in $\mathbf{H}^{-1}$. Moreover it holds \begin{align*}
            \E \left[\operatorname{sup}_{t\in[0,T]}\lVert \mathbf{M}(t)\rVert_{\mathbf{H}^{-1}}^2\right]\lesssim \kappa R_T^2T.
        \end{align*}
        \item[(ii)] For each $\epsilon\in (0,1/2), \ p\geq 1$
    \begin{align}\label{sc est1}
        \E \left[\operatorname{sup}_{t\in [0,T]}\lVert \mathbf{Z}(t)\rVert_{\mathbf{H}^{-\epsilon}}^p\right]^{1/p}\lesssim_{\epsilon,p,T}\sqrt{\kappa(\nu+\kappa)^{\epsilon-1}}R_T,
    \end{align}
    \begin{align}\label{sc est2}
        \E \left[\operatorname{sup}_{t\in [0,T]}\lVert \mathbf{Z}(t)\rVert_{\mathbf{H}^{-1-\epsilon}}^p\right]^{1/p}\lesssim_{\epsilon,p,T}\sqrt{\kappa(\nu+\kappa)^{\epsilon-1}}\lVert \theta \rVert_{\ell^{\infty}}R_T.
    \end{align}
    \item[(iii)] For $\beta\in (0,1]$ and $\epsilon\in (0,\beta]$, $p\geq 1$ it holds
    \begin{align}\label{sc est3}
        \E \left[\operatorname{sup}_{t\in [0,T]}\lVert \mathbf{Z}(t)\rVert_{\mathbf{H}^{-\beta}}^p\right]^{1/p}\lesssim_{\epsilon,p,T}\sqrt{\kappa(\nu+\kappa)^{\epsilon-1}}\lVert \theta \rVert_{\ell^{\infty}}^{\beta-\epsilon}R_T.
    \end{align}
    \end{itemize}
\end{lemma}
It is then easy to show that the weak solution $\q$ of the problem (\ref{eq:QG_deterministic GL}) satisfies also a mild formulation. In fact the following lemma holds.
\begin{lemma}\label{mild}
    If we denote by $\mathbf{G}(t):=\nu\mathbf{\Delta}M(-\mathbf{\Delta}-M)^{-1}\q(t)-\nabla^{\perp}\bfpsi(t)\cdot\nabla\q(t) -\beta \partial_x\bfpsi(t)+\textbf{F}(t)$, then 
    \begin{align*}
        \q(t)=e^{(k+\nu)t\mathbf{\Delta}}\q_0+\int_0^t e^{(k+\nu)(t-s)\mathbf{\Delta}}\mathbf{G}(s)ds+\mathbf{Z}(t),\ \ \mathbb{P}-a.s. \ \ \forall t\in[0,T].
    \end{align*}
\end{lemma}
\begin{proof}
    By definition, if $\q$ is a weak solution of problem \eqref{eq:QG_ito GL}, then it satisfies 
    \begin{align*}
        \langle \q(t),\bm{\phi}\rangle-\langle \q_0,\bm{\phi}\rangle=\int_0^t \langle(\kappa+\nu)\bm{\Delta}\q(s) +\mathbf{G}(s), \bm{\phi}\ \rangle_{\mathbf{H}^{-2},\mathbf{H}^2}ds-\sqrt{2\kappa}\sum_{k\in K}\int_0^t \langle\mathbf{a}_{k}e_k\cdot \nabla\bm{\phi},\q(s) \rangle dW^k_s
    \end{align*}
    $\mathbb{P}$-a.s. for each $t\in[0,T]$ and $\bm{\phi}\in D(-\bm{\Delta})$.
    If take $\bm{e}_i:=(e_i,e_i)^t$ as a test function, we have
    \begin{align*}
        \langle \q(t),\bm{e}_i\rangle-\langle \q_0,\bm{e}_i\rangle=\int_0^t \langle \lambda_i(\kappa+\nu)\q(s) +\mathbf{G}(s), \bm{e}_i \rangle_{\mathbf{H}^{-2},\mathbf{H}^2} ds-\sqrt{2\kappa}\sum_{k\in K}\int_0^t \langle\mathbf{a}_{k}e_k\cdot \nabla\bm{e}_i,\q(s) \rangle dW^k_s.
    \end{align*}
    If we apply It\^{o} formula to the process $e^{-t(\kappa+\nu)\lambda_i}\langle \q(t),\bm{e}_i\rangle$ and integrate in time we have
    \begin{align*}
        \langle \q(t),\bm{e}_i\rangle&= e^{-t(\kappa+\nu)\lambda_i}\langle \q_0,\bm{e}_i\rangle+\int_0^t e^{-(t-s)(\kappa+\nu)\lambda_i}\langle \mathbf{G}(s), \bm{e}_i \rangle ds\\ &-\sqrt{2\kappa}\sum_{k\in K}\int_0^t e^{-(t-s)(\kappa+\nu)\lambda_i} \langle\mathbf{a}_{k}e_k\cdot \nabla\bm{e}_i,\q(s) \rangle dW^k_s.
    \end{align*}
    We can then find $\Gamma\subseteq \Omega$ of full probability such that the above equality holds for all $t \in[0,T]$and all $i\in\mathbb{Z}^2_0$. But this is exactly the mild formulation written in Fourier modes.
\end{proof}
\begin{remark}\label{mild det}
    Similarly to Lemma~\ref{mild}, also the solution $\bar\q$ of the associated determinsitic equation \eqref{eq:QG_deterministic GL} satisfies for all $t\in [0,T]$ the integral relation
    \begin{align*}
        \bar\q(t)=e^{(k+\nu)t\mathbf{\Delta}}\q_0+\int_0^t e^{(k+\nu)(t-s)\mathbf{\Delta}}\left(\nu\mathbf{\Delta}M(-\mathbf{\Delta}-M)^{-1}\bar\q(s)-\nabla^{\perp}\bar\bfpsi(s)\cdot\nabla\bar\q(s) -\beta \partial_x\bar\bfpsi(s)+\bar{\textbf{F}}(s)\right)ds.
    \end{align*}
\end{remark}
Before proving Theorem~\ref{thm galeati}, we need a preliminary result on the nonlinearity of our problem.
\begin{lemma}\label{lipshitz nonlinear}
    Let  $\q\in \mathbf{L}^2,\ \bar{\q}\in \mathbf{H}^1$, then, given 
    \begin{align*}
      \bm{R}(\q)=-\nabla^{\perp}(-\mathbf{\Delta}-M)^{-1}\q\cdot \nabla \q,
    \end{align*}
    for each $\alpha\in (0,1)$ the following relation holds true
    \begin{align*}
        \lVert \bm{R}(\q)-\bm{R}(\bar{\q})\rVert_{\mathbf{H}^{-1-\alpha}}\lesssim_{\alpha,M}\lVert \q-\bar{\q}\rVert_{\mathbf{H}^{-\alpha}}\left(\lVert \q\rVert+\lVert \bar{\q}\rVert_{\mathbf{H}^1}\right).
    \end{align*}
\end{lemma}
\begin{proof}
    With a simple manipulation and the triangular inequality we have
    \begin{align*}
            \lVert \bm{R}(\q)-\bm{R}(\bar{\q})\rVert_{\mathbf{H}^{-1-\alpha}}& \leq   \lVert \nabla^{\perp}(-\mathbf{\Delta}-M)^{-1}\bar{\q}\cdot \nabla (\q-\bar{\q}) \rVert_{\mathbf{H}^{-1-\alpha}}+ \lVert \nabla^{\perp}(-\mathbf{\Delta}-M)^{-1}(\q-\bar{\q})\cdot \nabla \q \rVert_{\mathbf{H}^{-1-\alpha}}\\
            &=: I_1 + I_2.
    \end{align*}
    Then the thesis follows by Lemma~\ref{preliminary non linear}. In fact by point 2, we have  
    \begin{align*}
        I_1\lesssim_{\alpha, M} \lVert \nabla^{\perp}(-\mathbf{\Delta}-M)^{-1}\bar{\q} \rVert_{\mathbf{H}^2}\lVert \bar\q-\q \rVert_{\mathbf{H}^{-\alpha}}\lesssim \lVert \bar\q\rVert_{\mathbf{H}^{1}} \lVert \bar\q-\q \rVert_{\mathbf{H}^{-\alpha}},
    \end{align*}
    and by point $4$ it follows 
    \begin{align*}
        I_2\lesssim_{\alpha,M} \lVert \q\rVert \lVert \nabla^{\perp}(-\mathbf{\Delta}-M)^{-1}(\q-\bar{\q}) \rVert_{\mathbf{H}^{1-\alpha}}\lesssim_{\alpha,M} \lVert \q\rVert \lVert \q-\bar{\q}\rVert_{\mathbf{H}^{-\alpha}}.
    \end{align*}
\end{proof}
We are now ready to show the main result of this section, \autoref{thm galeati}:
\begin{proof}[Proof of Theorem~\ref{thm galeati}]
    We have seen in Lemma~\ref{mild} and Remark \ref{mild det} that both $\q$ and $\bar{\q}$ satisfies a mild formulation. Thus calling $\bm{\xi}=\q-\bar{\q}$ and $\bm{\chi}=\bfpsi-\bar{\bfpsi}$ we have
    \begin{multline*}
        \bm{\xi}(t)=\int_0^t e^{(\kappa+\nu)(t-s)\bm{\Delta}}\left(\nu\mathbf{\Delta}M(-\mathbf{\Delta}-M)^{-1}\bm{\xi}(s)-\beta \partial_x\bm{\chi}(s)+{\textbf{F}}(s)-\bar{\textbf{F}}(s)\right)ds\\ -\int_0^t e^{(\kappa+\nu)(t-s)\bm{\Delta}}\left(\bm{R}(\q(s))-\bm{R}(\bar{\q}(s))\right)ds+\bm{Z}(t).
    \end{multline*}
    By \autoref{lemma:12_2.3} we have 
    \begin{multline*}
         \lVert\bm{\xi}(t)\rVert_{\mathbf{H}^{-\alpha}}^2 \lesssim_{\alpha,M} \lVert \bm{Z}(t)\rVert_{\mathbf{H}^{-\alpha}}^2 +\frac{1}{\kappa+\nu}\int_0^t\lVert \bm{R}(\q(s))-\bm{R}(\bar{\q}(s)) \rVert_{\mathbf{H}^{-\alpha-1}}^2\ ds \\  + \frac{1}{\kappa+\nu}\int_0^t  \nu^2\|\mathbf{\Delta}M(-\mathbf{\Delta}-M)^{-1}\bm{\xi}(s)\|^2_{H^{-\alpha -1}} + \beta^2\| \partial_x\bm{\chi}(s)\|^2_{\bH^{-\alpha -1}} + r^2\|\Delta\chi_2(s)\|^2_{H^{-\alpha -1}}\, ds 
    \end{multline*}
    and by the relation \eqref{eq:relation_q_psi} it follows
    \begin{equation*}
        \lVert\bm{\xi}(t)\rVert_{\mathbf{H}^{-\alpha}}^2 \lesssim_{\alpha,M} \lVert \bm{Z}(t)\rVert_{\mathbf{H}^{-\alpha}}^2+\frac{1}{\kappa+\nu}\int_0^t\lVert \bm{R}(\q(s))-\bm{R}(\bar{\q}(s)) \rVert_{\mathbf{H}^{-\alpha-1}}^2\ ds  + \frac{\beta^2+\nu^2+r^2}{\kappa+\nu}\int_0^t \lVert \bm{\xi}(s) \rvert_{\mathbf{H}^{-\alpha-1}}^2\ ds.
    \end{equation*}
    Last, thanks to \autoref{lipshitz nonlinear}
    \begin{align*}
        \lVert\bm{\xi}(t)\rVert_{\mathbf{H}^{-\alpha}}^2
        & \lesssim_{\alpha,M}\lVert \bm{Z}(t)\rVert_{\mathbf{H}^{-\alpha}}^2+ \frac{1}{\kappa+\nu}\int_0^t \lVert \bm{\xi}(s)\rVert_{\mathbf{H}^{-\alpha}}^2\left(\lVert \q(s)\rVert^2+\lVert \bar{\q}(s)\rVert_{V}^2+\beta^2+\nu^2+r^2\right)\ ds.
    \end{align*}
    Therefore, by Gronwall's lemma, there exists $C=C(\alpha,M)$ such that
    \begin{align}\label{Gronwall estimate}
         \lVert\bm{\xi}(t)\rVert_{\mathbf{H}^{-\alpha}}^2\lesssim_{\alpha,M} \left(\operatorname{sup}_{t\in[0,T]}\lVert \bm{Z}(t)\rVert_{\mathbf{H}^{-\alpha}}^2\right)\exp\left(\frac{C}{\nu+\kappa}\int_0^T \lVert \q(s)\rVert^2+\lVert \bar\q(s)\rVert_V^2\ ds\right)\exp\left(T\frac{\nu^2+\beta^2+r^2}{\kappa+\nu}\right).
    \end{align}
    Now we take the expectation of \eqref{Gronwall estimate} and use relation (\ref{sc est3}) in \autoref{regularity stochastic convolution} to estimate the stochastic convolution to get 
    \begin{equation}
    \label{eq:E Growall estimate}
         \E\, \lVert\bm{\xi}(t)\rVert_{\mathbf{H}^{-\alpha}}^2 \lesssim_{\alpha,M, \epsilon, T} \frac{\kappa}{ (\nu + \kappa)^{1- \epsilon}} \|\theta\|_{\ell^\infty}^{2(\alpha - \epsilon)} R_T^2 \exp\left(T\frac{\nu^2+\beta^2+r^2}{\kappa+\nu}\right) \E\, \exp\left(\frac{C}{\nu+\kappa}\int_0^T \lVert \q(s)\rVert^2+\lVert \bar\q(s)\rVert_V^2\ ds\right).
    \end{equation}
    Now to prove statement (i) we use \eqref{eq:defR_T} and \eqref{eq:int nabla q bar} to derive
    \begin{align*}
        \E \left[\lVert \q-\bar\q\rVert_{C([0,T];\mathbf{H}^{-\alpha})}^2\right]\lesssim_{\alpha,M,\epsilon,T}& \kappa^{\epsilon}\lVert \theta\rVert_{\ell^{\infty}}^{2(\alpha-\epsilon)}R_T^2\exp\left(T\frac{\nu^2+\beta^2+r^2}{\kappa+\nu}\right)\\ & \exp\left(\frac{CTR_T^2}{(\kappa+\nu)^2}\left(1+\kappa+\nu\right)+\frac{C}{\left(\kappa+\nu\right)^2}\int_0^T\lVert F(s)\rVert^2\ ds\right).
    \end{align*}
    
    We move on to proving statement (ii). Given \eqref{eq:E Growall estimate} let us now estimate $\int_0^T\| \q(s)\|^2$ using the estimate \eqref{eq:int nabla q} instead of \eqref{eq:defR_T}, namely
    \begin{align*}
        \int_0^T\lVert \q (s)\rVert^2 ds\lesssim \int_0^T\lVert \q (s)\rVert_V^2 ds\lesssim  \frac{C}{\nu} \left(TR_T^2+\int_0^T\lVert F(s)\rVert^2\ ds\right)
    \end{align*}
    where we have also used that $\frac{1}{\nu}+\frac{1}{\kappa+\nu}\leq \frac{2}{\nu}$. Then the desired bound follows from the same arguments used for the estimate (i).
\end{proof}
\begin{remark}
    Similarly to \cite{flandoli2021quantitative}, we can consider also the $p$ moment of the random variables treated in the previous theorem. We neglect this fact, which is not needed in order to prove Theorem~\ref{thm:convergence deterministic}.
\end{remark}

\subsection{Long Time Behavior}\label{sec long time}
%
Let us recall our framework, $\q$ is the weak solution of the system  
\begin{align*}
    \begin{cases}
    dq_1 + \left(\nabla^{\perp}\psi_1\cdot \nabla q_1\right) dt = \left(\kappa\Delta q_1+\nu\Delta^2 \psi_1 - \beta \partial_x \psi_1 + F \right)\, dt + \sqrt{2\kappa}\sum_{k=1}^K\bm{a}_{1,k}e_k\cdot \nabla q_1\,  dW^{1,k}\\
    dq_2 + \left(\nabla^{\perp}\psi_2\cdot \nabla q_2\right) dt = \left(\kappa\Delta q_2 + \nu\Delta^2 \psi_2 -\beta \partial_x \psi_2 - r\Delta \psi_2\right)\, dt + \sqrt{2\kappa}\sum_{k=1}^K\bm{a}_{2,k}e_k\cdot \nabla q_2 \, dW^2 \\
\bfpsi=-(-\mathbf{\Delta}-M)^{-1}\q\\
\q(0)=\q_0
\end{cases}
\end{align*}
and $\tilde{\q}$ is a weak solution of the stationary deterministic system
\begin{align*}
    \begin{cases}
        \nabla^{\perp}\tilde{\psi}_1\cdot \nabla \tilde{q}_1 = \kappa\Delta \tilde{q}_1+\nu\Delta^2 \tilde{\psi}_1 - \beta \partial_x \tilde{\psi}_1 + F\\
    \nabla^{\perp}\tilde{\psi}_2\cdot \nabla \tilde{q}_2 = \kappa\Delta \tilde{q}_2+\nu\Delta^2 \tilde{\psi}_2 -\beta \partial_x \tilde{\psi}_2 - r\Delta \tilde{\psi}_2\\
    \tilde{\bfpsi}=-(-\mathbf{\Delta}-M)^{-1}\tilde{\q}.
    \end{cases}
\end{align*}
Thanks to the relation between $\tilde{\q}$ and $\tilde{\bfpsi}$, the stationary system can be rewritten as 
\begin{align}\label{stationary}
\begin{cases}
    0=(\kappa+\nu)\mathbf{\Delta} \tilde{\q}+\nu\mathbf{\Delta}M(-\mathbf{\Delta}-M)^{-1}\tilde{\q}-\nabla^{\perp}\tilde{\bfpsi}\cdot\nabla\tilde{\q}-\beta\partial_x\tilde{\bfpsi}+\Tilde{\mathbf{F}}\\
    \tilde{\bfpsi}=-(-\mathbf{\Delta}-M)^{-1}\tilde{\q},\\  
\end{cases}
\end{align}
where 
$$\Tilde{\mathbf{F}}=\begin{bmatrix}
    F\\ -r\Delta\tilde{\psi}_2
    \end{bmatrix}.$$
    Then we consider the following concept of solution for \eqref{stationary}:
\begin{definition}\label{weak stationary solution}
    A weak solution of the problem (\ref{stationary}) is a function $\tilde{\q}\in  \cV$ such that $\forall \bm{\phi}=(\phi_1,\phi_2)^t \in  \cV$ the following relation holds true
    \begin{align*}
        (\kappa+\nu)\langle \nabla \tilde{\q},\nabla\bm{\phi}\rangle=\nu\langle\mathbf{\Delta} M(-\mathbf{\Delta}-M)^{-1}\tilde{\q},\bm{\phi}\rangle-\langle \nabla^{\perp}\tilde{\bm{\psi}}\cdot\nabla\tilde{\q}, \bm{\phi}\rangle-\beta\langle \partial_x \tilde{\bm{\psi}},\bm{\phi}\rangle+\langle \tilde{\mathbf{F}},\bm{\phi}\rangle.
    \end{align*}
\end{definition}    
\begin{proposition}\label{lemma convergence}
    For $\kappa$ large enough there exists a unique $\tilde{\q}$ weak solution of problem (\ref{stationary}), moreover $\lVert\Tilde{\q}-\Bar\q(t)\rVert\rightarrow 0$ exponentially fast as $t\rightarrow+\infty$.
\end{proposition}
\begin{proof}
    As usual we start by establishing a priori estimates. 
    Assuming there exist a solution $\tilde\q$. Taking $\tilde\q$ as a test function in the weak formulation, we obtain, thanks to the fact that $\langle \tilde{\bm\psi}\cdot\nabla\tilde \q, \tilde{\q}\rangle=0$ and to Poincaré inequality,
    \begin{align*}
        (\kappa+\nu)\lVert \nabla\tilde \q\rVert^2\leq C\nu \lVert \nabla\tilde \q\rVert^2+C\lVert F\rVert \lVert \nabla\tilde \q\rVert+C \lVert \nabla\tilde \q\rVert^2
    \end{align*}
    where the constant $C$ depends from $\beta, S_1,\ S_2, r, \mathcal{D}$ but it is independent of $\kappa$ and $\nu$. Thus, if $\kappa$ is large enough we have
    \begin{align}\label{apriori}
        \lVert \nabla\tilde \q\rVert\leq \frac{C\lVert F\rVert}{\kappa+\nu-C\nu-C}=M_{\kappa}.
    \end{align}
    %
    For $\kappa$ large enough such that the a priori estimate \ref{apriori} holds, let us consider the complete metric space 
    $$ B_{\kappa}=\{\q\in  \cV: \lVert \q\rVert_V\leq M_{\kappa}\}.$$ 
    Let us show that the map $T$ which takes $\q \in B_{\kappa}$ and associate to it $\tilde{\q}$ which is the unique weak solution of the linear problem
    \begin{align*}
        (\kappa+\nu)\langle \nabla \tilde{\q},\nabla\bm{\phi}\rangle=\nu\langle\mathbf{\Delta} M(-\mathbf{\Delta}-M)^{-1}{\q},\bm{\phi}\rangle-\langle \nabla^{\perp}{\bm{\psi}}\cdot\nabla\tilde{\q}, \bm{\phi}\rangle-\beta\langle \partial_x {\bm{\psi}},\bm{\phi}\rangle+\langle {\mathbf{F}},\bm{\phi}\rangle.
    \end{align*}
    is a contraction in $B_{\kappa}$ if we take $\kappa$ large enough. Here, of course $\bm{\psi}$ is obtained by $\q$ via relation (\ref{eq:relation_q_psi}) and $${\mathbf{F}}=\begin{bmatrix}
        F\\ -r\Delta{\psi}_2
        \end{bmatrix}.$$ Existence and uniqueness of the solution of previous problem follows immediately by Lax-Milgram Lemma thanks to the fact that the bilinear form $a: \cV\times  \cV\rightarrow \mathbb{R}$ defined by  $$a(\q_1,\q_2)=(\kappa+\nu)\langle \nabla \q_1,\nabla \q_2\rangle-\langle \nabla^{\perp}{\bm\psi}\cdot\nabla \q_1, \q_2\rangle$$ is continuous and coercive. Moreover, arguing as in the a priori estimate, $T(\q)$ satisfies \begin{align*}
        (\kappa+\nu)\lVert \nabla T(\q)\rVert\leq C\nu M_{\kappa}+CM_{\kappa}+C\lVert  F\rVert
    \end{align*}
    where $C$ is the same costant as above. From this it follows immediately that
    \begin{align*}
        \lVert \nabla T(\q)\rVert\leq M_{\kappa}.
    \end{align*}
    Thus $T$ is a map between $B_{\kappa}$ and itself for $\kappa$ large enough such that $M_{\kappa}>0$. Lastly we need to show that $T$ is a contraction. Let $\q_1,\q_2\in B_{\kappa}$ and $T(\q_1)$, $T(\q_2)$ the corresponding solutions. Then $\forall \bm{\phi} \in  \cV$ we have
    \begin{align*}
        (\kappa+\nu)\langle \nabla (T(\q_1)-T(\q_2)),\nabla\bm\phi\rangle&=\nu\langle\mathbf{\Delta} M(-\mathbf{\Delta}-M)^{-1}(\q_1-\q_2),\bm\phi\rangle-\langle \nabla^{\perp}\bm\psi_1\cdot\nabla T(\q_1), \bm\phi\rangle\\ &+\langle \nabla^{\perp}\bm\psi_2\cdot\nabla T(\q_2), \bm\phi\rangle-\beta\langle \partial_x (\bm\psi_1-\bm\psi_2),\bm\phi\rangle+\langle {\bm F}_1-{\bm F}_2,\bm \phi\rangle.
    \end{align*}
    Here, of course $\bm{\psi}_{1,2}$ is obtained by $\q_{1,2}$ via relation (\ref{eq:relation_q_psi}) and $${\mathbf{F}_{1,2}}=\begin{bmatrix}
        F\\ -r(\Delta{\bm\psi}_{1,2})_2
        \end{bmatrix}.$$
    Taking $\bm\phi=T(\q_1)-T(\q_2)$ we have 
    \begin{align*}
      (\kappa+\nu)\lVert \nabla ( T(\q_1)-T(\q_2))\rVert^2& \leq C\nu \lVert \nabla  (T(\q_1)-T(\q_2))\rVert \lVert \nabla  (\q_1-\q_2)\rVert +C\lVert \nabla  (T(\q_1)-T(\q_2))\rVert \lVert \nabla  (\q_1-\q_2)\rVert\\ & -\langle \nabla^{\perp}\bm\psi_1\cdot\nabla T(\q_1), T(\q_1)-T(\q_2)\rangle+\langle \nabla^{\perp}\bm\psi_2\cdot\nabla T(\q_2), T(\q_1)-T(\q_2)\rangle\\ & \pm  \langle \nabla^{\perp}\bm\psi_1\cdot\nabla T(\q_2), T(\q_1)-T(\q_2)\rangle\\ & \leq C\nu \lVert \nabla  (T(\q_1)-T(\q_2))\rVert \lVert \nabla  (\q_1-\q_2)\rVert +C\lVert \nabla  (T(\q_1)-T(\q_2))\rVert \lVert \nabla  (\q_1-\q_2)\rVert\\ &+C M_{\kappa}\lVert \nabla  (T(\q_1)-T(\q_2))\rVert \lVert \nabla  (\q_1-\q_2)\rVert.
    \end{align*}
    If we take $\kappa$ large enough, thus $T$ is a contraction and the thesis follows.
    %
    
    We conclude by showing the desired exponential convergence. Let $\bar{\bfpsi}$ and $\tilde{\bfpsi}$ be the stream functions associated to $\bar{\q}$ and $\tilde{\q}$ respectively and define $\bfw =\bar\q-\tilde\q$ and $\bm{\chi}=\bar\bfpsi-\tilde\bfpsi$. 
    The following differential relation holds 
    \begin{equation*}
         \frac{d\lVert \bfw \rVert^2}{2dt}+(\kappa+\nu)\lVert \nabla \bfw \rVert^2
         =\langle \nu \Delta M(-\Delta-M)^{-1}\bfw ,\bfw \rangle-\beta\langle \partial_x \bm{\chi},\bfw \rangle+\langle \Bar{\textbf{F}}-\tilde{\textbf{F}},\bfw \rangle -\langle\nabla^{\perp}\bm{\chi}\cdot\nabla  \tilde\q,\bfw \rangle.
    \end{equation*}
    %
    %
    Using the definition of $\bar{\bm{F}}$ and $\tilde{\bm{F}}$, \autoref{preliminary non linear} and Young's inequality we have 
    \begin{align*}
       \frac{d\lVert \bfw \rVert^2}{2dt}+(\kappa+\nu)\lVert \nabla \bfw \rVert^2 & \leq \nu C\lVert \bfw \rVert^2+C\lVert \bfw\rVert^2+C\lVert \bfw\rVert\lVert \nabla \bfw\rVert\lVert \nabla \tilde{\q}\rVert\\ & 
        \leq \nu C\lVert \bfw \rVert^2+C\lVert \bfw \rVert^2+\frac{\nu}{2}\lVert \nabla \bfw \rVert^2+\frac{C}{2\nu}\lVert \bfw \rVert^2\lVert \nabla \tilde{\q}\rVert^2,
    \end{align*}
    where $C$ is a constant possibly depending from $\mathcal{D}$, $\beta,\ r,\ M$ but it is independent of $\kappa$, $\nu$ and $T$. By Poincaré inequality, we have $\lVert \nabla \bfw \rVert^2\geq \frac{1}{C_p}\lVert \bfw \rVert^2$, therefore \begin{align*}
        \frac{d \lVert \bfw \rVert^2}{2 dt}+\left(\frac{\kappa+\nu/2}{C_p}-(\nu C+C+C\lVert \nabla\tilde \q\rVert^2)\right)\lVert \bm{w}\rVert^2\leq 0.
    \end{align*}
    Calling $\alpha=\frac{\kappa+\nu/2}{C_p}-(\nu C+C+C\lVert \nabla\tilde \q\rVert^2)$. If $\kappa$ is large enough, $\alpha>0$ thus by Gronwall's Lemma we have the exponential rate of convergence.
\end{proof}

\begin{remark}
\begin{enumerate}
    \item By the proof of previous theorem, $\kappa$ must be large enough that the following inequalities are satisfied, here $C$ is a constant possibly depending from $\beta, S_1, S_2, r, \mathcal{D}$ but independent of $\kappa$ and $\nu$.
    \begin{align*}
       & \kappa+\nu-C\nu-C>0\\
       &\frac{1}{\kappa+\nu}\left(C\nu+C+\frac{C\lVert F\rVert}{\kappa+\nu-C\nu-C}\right)<1\\
       & \frac{\kappa+\nu/2}{C_p}-\left(\nu C+C+\frac{C\lVert F\rVert^2}{\left(\kappa+\nu-C\nu-C\right)^2}\right)>0.
    \end{align*}
    \item By elliptic regularity, it follows immediately that, actually, $\tilde\q\in D(\mathbf{\Delta}).$
\end{enumerate}
\end{remark}

Now we are able to prove our final result.
\begin{proof}[Proof of Theorem~\ref{thm:convergence deterministic}]
    Let $\bar{\q}$ be the weak solution of the stationary deterministic problem (\ref{eq:QG_deterministic GL}). First we fix $\delta>0,\ \alpha\in (0,1)$. If $\kappa$ is large enough, by \autoref{lemma convergence}, we can find $\Bar{T}=\Bar{T}(\delta)$ such that
    \begin{align*}
        \lVert \Bar{\q}(t)-\Tilde{\q}\rVert^2\leq \delta/4,\ \fa t\geq \Bar{T}.
    \end{align*}
    Now we use the results of Theorem~\ref{thm galeati} for $\epsilon=\alpha/2$, thus we have
    \begin{align*}
        \E \left[\lVert \q-\bar\q\rVert_{C([0,2\bar{T}];\mathbf{H}^{-\alpha})}^2\right] \lesssim_{\alpha,M,\epsilon,\bar{T}}
        & \kappa^{\epsilon}\lVert \theta\rVert_{\ell^{\infty}}^{2(\alpha-\epsilon)}R_{2\bar{T}}^2 \exp\left(2\bar{T}\frac{\nu^2+\beta^2+r^2}{\kappa+\nu}\right)\\
        & \exp\left(\frac{C\bar{T}R_{2\bar{T}}^2}{(\kappa+\nu)^2}\left(1+\kappa+\nu\right)+\frac{C}{\left(\kappa+\nu\right)^2}\int_0^{2\bar{T}}\lVert F(s)\rVert^2\ ds\right).
    \end{align*}
    Since the constants appearing in previous equation are independent of the parameters of the noise, if we take $\theta$ to be such that the right hand side of the previous inequality can be bounded by $\delta/4$ then the thesis follows immediately. For example some possible choices of $\theta$ can be found in \cite[Example~1.3]{flandoli2021quantitative}.
\end{proof}

\paragraph{Acknowledgments} 
The authors are grateful to Franco Flandoli and Valerio Lucarini for
criticisms, comments, suggestions, and encouragement. In particular, GC would like to thank professor Flandoli for hosting and financially contributing to her research visit at Scuola Normale di Pisa in May 2022, where this work started.

\paragraph{} Dipartimento di Ingegneria e Scienze dell’Informazione e Matematica, Universit\`{a} degli Studi dell’Aquila, 67100 L’Aquila, Italy\\
Centre for the Mathematics of Planet Earth, University of Reading, United Kingdom\\
\textit{Email:} giulia.carigi@univaq.it

\paragraph{} Scuola Normale Superiore, Piazza dei Cavalieri,7, 56126 Pisa, Italy \\ \textit{Email:} eliseo.luongo@sns.it


\begin{thebibliography}{10}

\bibitem{Bernier94}
C.~Bernier.
\newblock {Existence of attractor for the quasi-geostrophic approximation of
  the Navier-Stokes equations and estimate of its dimension}.
\newblock {\em Adv. Math. Sci. Appl.}, 4(2):465--489, 1994.

\bibitem{breckner1999approximation}
H.~Breckner.
\newblock Approximation and optimal control of the stochastic navier-stokes
  equation.
\newblock {\em Mathematisch Naturwissenschaftlich Technischen Fakult{\"a}t der
  Martin Luther Universit{\"a}t Halle Wittenberg, Diss}, 1999.

\bibitem{breckner2000galerkin}
H.~Breckner.
\newblock Galerkin approximation and the strong solution of the navier-stokes
  equation.
\newblock {\em Journal of Applied Mathematics and Stochastic Analysis},
  13(3):239--259, 2000.

\bibitem{thesis}
G.~Carigi.
\newblock {\em Ergodic properties and response theory for a stochastic
  two-layer model of geophysical fluid dynamics.}
\newblock PhD thesis, University of Reading, 2021.
\newblock https://doi.org/10.48683/1926.00102181.

\bibitem{carigi2022exponential}
G.~Carigi, J.~Bröcker, and T.~Kuna.
\newblock Exponential ergodicity for a stochastic two-layer quasi-geostrophic
  model.
\newblock {\em arXiv preprint arXiv:2201.09823}, 2022.

\bibitem{Charney48}
J.~G. Charney.
\newblock On the scale of atmospheric motions.
\newblock {\em Geofysiske Publikasjoner}, 17(2):1--17, 1948.

\bibitem{Chueshov01_Proba}
I.~Chueshov, J.~Duan, and B.~Schmalfuss.
\newblock Probabilistic dynamics of two-layer geophysical flows.
\newblock {\em Stochastics and Dynamics}, 01(04):451--475, 2001.

\bibitem{cotter2020data}
C.~Cotter, D.~Crisan, D.~Holm, W.~Pan, and I.~Shevchenko.
\newblock Data assimilation for a quasi-geostrophic model with
  circulation-preserving stochastic transport noise.
\newblock {\em Journal of Statistical Physics}, pages 1--36, 2020.

\bibitem{cotter2020modelling}
C.~Cotter, D.~Crisan, D.~Holm, W.~Pan, and I.~Shevchenko.
\newblock Modelling uncertainty using stochastic transport noise in a 2-layer
  quasi-geostrophic model.
\newblock {\em Foundations of Data Science}, 2(2):173, 2020.

\bibitem{DaPratoZab2014}
G.~Da~Prato and J.~Zabczyk.
\newblock {\em Stochastic equations in infinite dimensions}.
\newblock Cambridge university press, 2014.

\bibitem{dong2022dissipation}
Z.~Dong, D.~Luo, and B.~Tang.
\newblock Dissipation enhancement by transport noise for stochastic $ p
  $-laplace equations.
\newblock {\em arXiv preprint arXiv:2206.01376}, 2022.

\bibitem{flandoli2021quantitative}
F.~Flandoli, L.~Galeati, and D.~Luo.
\newblock Quantitative convergence rates for scaling limit of spdes with
  transport noise.
\newblock {\em arXiv preprint arXiv:2104.01740}, 2021.

\bibitem{flandoli2022eddy}
F.~Flandoli, L.~Galeati, and D.~Luo.
\newblock Eddy heat exchange at the boundary under white noise turbulence.
\newblock {\em Philosophical Transactions of the Royal Society A},
  380(2219):20210096, 2022.

\bibitem{Flandoli_2021}
F.~Flandoli and D.~Luo.
\newblock High mode transport noise improves vorticity blow-up control in 3d
  navier{\textendash}stokes equations.
\newblock {\em Probability Theory and Related Fields}, 180(1-2):309--363, mar
  2021.

\bibitem{flandoli2022heat}
F.~Flandoli and E.~Luongo.
\newblock Heat diffusion in a channel under white noise modeling of turbulence.
\newblock {\em Mathematics in Engineering}, 4(4):1--21, 2022.

\bibitem{flaWas}
F.~Flandoli and E.~Luongo.
\newblock {\em Stochastic Partial Differential Equations in Fluid Mechanics}.
\newblock Springer, to appear.

\bibitem{FlandoliPappalettera2022}
F.~Flandoli and U.~Pappalettera.
\newblock From additive to transport noise in 2d fluid dynamics.
\newblock {\em Stochastics and Partial Differential Equations: Analysis and
  Computations}, 2022.

\bibitem{galeati2020convergence}
L.~Galeati.
\newblock On the convergence of stochastic transport equations to a
  deterministic parabolic one.
\newblock {\em Stochastics and Partial Differential Equations: Analysis and
  Computations}, 8(4):833--868, 2020.

\bibitem{Holm_2015}
D.~D. Holm.
\newblock Variational principles for stochastic fluid dynamics.
\newblock {\em Proceedings of the Royal Society A: Mathematical, Physical and
  Engineering Sciences}, 471(2176):20140963, 2015.

\bibitem{karatzas2012brownian}
I.~Karatzas and S.~Shreve.
\newblock {\em Brownian motion and stochastic calculus}, volume 113.
\newblock Springer Science \& Business Media, 2012.

\bibitem{Lucarinietal2014}
V.~Lucarini, R.~Blender, C.~Herbert, F.~Ragone, S.~Pascale, and J.~Wouters.
\newblock Mathematical and physical ideas for climate science.
\newblock {\em Reviews of Geophysics}, 52(4):809--859, 2014.

\bibitem{luongo2022inviscid}
E.~Luongo.
\newblock Inviscid limit for stochastic second-grade fluid equations.
\newblock {\em arXiv preprint arXiv:2207.03174}, 2022.

\bibitem{pardoux1975equations}
E.~Pardoux.
\newblock Equations aux d{\'e}riv{\'e}es partielles stochastiques monotones,
  these, univ, 1975.

\bibitem{pazy2012semigroups}
A.~Pazy.
\newblock {\em Semigroups of linear operators and applications to partial
  differential equations}, volume~44.
\newblock Springer Science \& Business Media, 2012.

\bibitem{Pedlosky13}
J.~Pedlosky.
\newblock {\em Geophysical fluid dynamics}.
\newblock Springer Science \& Business Media, 2013.

\bibitem{razafimandimby2012strong}
P.~A. Razafimandimby and M.~Sango.
\newblock Strong solution for a stochastic model of two-dimensional second
  grade fluids: existence, uniqueness and asymptotic behavior.
\newblock {\em Nonlinear Analysis: Theory, Methods \& Applications},
  75(11):4251--4270, 2012.

\bibitem{skorokhod1982studies}
A.~V. Skorokhod.
\newblock {\em Studies in the theory of random processes}, volume 7021.
\newblock Courier Dover Publications, 1982.

\bibitem{Vallis06}
G.~K. Vallis.
\newblock {\em Atmospheric and Oceanic Fluid Dynamics}.
\newblock Cambridge University Press, Cambridge, U.K., 2006.

\end{thebibliography}
\end{document}